\documentclass[12pt]{amsart}

\usepackage[latin1]{inputenc}					
\usepackage{amsmath,amsfonts,amssymb,amsthm}	

\usepackage{geometry}		
\usepackage{mathtools}		
\usepackage{hyperref}		
\usepackage{setspace}		
\usepackage{cite}			

\hypersetup{colorlinks=true,citecolor=blue,linkcolor=blue,urlcolor=magenta}

\numberwithin{equation}{section}			
\newtheorem{theorem}{Theorem}
\newtheorem{proposition}{Proposition}
\newtheorem{lemma}{Lemma}
\newtheorem{corollary}{Corollary}

\newtheorem{remark}{Remark}
\newtheorem{definition}{Definition}

\newtheorem*{theorem*}{Theorem}
\newtheorem*{proposition*}{Proposition}
\newtheorem*{lemma*}{Lemma}
\newtheorem*{corollary*}{Corollary}

\newtheorem*{conjecture*}{Conjecture}
\newtheorem*{question*}{Question}
\newtheorem*{remark*}{Remark}
\newtheorem*{definition*}{Definition}

\renewcommand{\leq}{\leqslant}	
\renewcommand{\geq}{\geqslant}	

\newcommand{\N}{\mathbb{N}}		
\newcommand{\Z}{\mathbb{Z}}		
\newcommand{\Q}{\mathbb{Q}}		
\newcommand{\R}{\mathbb{R}}		
\newcommand{\C}{\mathbb{C}}		
\newcommand{\T}{\mathbb{T}}		

\newcommand{\eps}{\varepsilon}			
\DeclareMathOperator{\re}{Re}			

\newcommand{\wt}{\widetilde}		
\newcommand{\wh}{\widehat}			

\newcommand{\E}{\mathbb{E}}			
\newcommand{\p}{\mathbb{P}}			

\DeclareMathOperator{\Ker}{Ker}			
\DeclareMathOperator{\rk}{rk}			

\newcommand{\M}{\mathcal{M}}			
\newcommand{\bx}{\mathbf{x}}			
\newcommand{\by}{\mathbf{y}}			
\newcommand{\bbA}{\mathbb{A}}			
\newcommand{\bbK}{\mathbb{K}}			
\newcommand{\bB}{\mathbf{B}}			

\newcommand{\isom}{\xrightarrow{
   \,\smash{\raisebox{-0.25ex}{\ensuremath{\scriptstyle\sim}}}\,}}

\newtheorem{fact}{Fact}
\newtheorem*{fact*}{Fact}

\newcommand{\transp}[1]{\prescript{\mathrm{t}}{}{#1}}

\begin{document}

\title{On systems of complexity one in the primes}

\author{Kevin Henriot}

\date{}

\begin{abstract}
	Consider a translation-invariant
	system of linear equations $V \bx = 0$
	of complexity one,
	where $V$ is an integer $r \times t$ matrix.
	We show that if $A$ is a subset of
	the primes up to $N$ of density at
	least $C(\log\log N)^{-1/25t}$, there
	exists a solution $\bx \in A^t$ 
	to $V \bx = 0$ with distinct coordinates.
	This extends a quantitative result of 
	Helfgott and de Roton
	for three-term arithmetic progressions,
	while the qualitative result is known to hold 
	for all translation-invariant systems of 
	finite complexity by the work
	of Green and Tao.
\end{abstract}

\maketitle

\section{Introduction}
\label{sec:intro}

Consider a matrix $V \in \M_{r \times t}(\Z)$
with coefficients on each line summing to $0$,
a condition we term \textit{translation-invariant}.
We are interested in special instances of the problem of 
finding a distinct-coordinates 
solution $\by \in A^t$ to the system of equations
$V \by = 0$, where $A$ is a dense subset 
of the set $\mathcal{P}_N$ of the primes up to a large integer $N$, 
and when the relative density decays with $N$.
Note that the distinct-coordinates condition excludes
trivial solutions of the form $(u,\dots,u)$,
while the conditions of homogeneity and translation-invariance
on the system of equations are necessary to expect 
a Szemerédi-type theorem for $V \by = 0$,
as can be seen by examining the case of a single linear equation
(see e.g.~\cite[Theorem~1.3]{Ruzsa:lineqs}).

We may assume that $V$ has rank $r$ up to
removing redundant equations.
Furthermore, we may work in practice with a parametrization 
$\psi: \Z^{t-r} \isom \Z^t \cap \Ker(V)$,
and look instead for occurences of 
distinct-coordinates values of $\psi$ in $A^t$.
The canonical setting of study
is that of the single translation-invariant equation 
${y_1 + y_3 = 2y_2}$, which detects $3$-term arithmetic progressions, 
themselves parametrized by the system of forms
\begin{align*}
	\psi(x_1,x_2) = (x_1, x_1 + x_2, x_1 + 2x_2).
\end{align*}
It is then a well-known result of Green~\cite{G:rothprimes}
that every subset of $\mathcal{P}_N$ of positive density
contains a non-trivial three-term arithmetic progression;
and the extension of this result to progressions of any length
is the celebrated Green-Tao theorem~\cite{GT:apsinprimes}.
Green's argument~\cite{G:rothprimes} 
actually allowed for densities as low as 
$(\log\log\log\log N)^{-1/2+o(1)}$,
and Helfgott and de Roton~\cite{HR:rothprimes}
later obtained a remarkable quantitative strenghtening
of this result.
\begin{theorem}[Helfgott, de Roton]
	Suppose that $A$ is a subset of $\mathcal{P}_N$
	of density at least\footnote{Throughout this introduction, we write $(\log_k N)^{o(1)}$ 
	for unspecified factors of the form $C(\log_{k+1} N)^C$ with $C >0$,
	where $\log_k$ is the $k$-th iterated logarithm.}
	\begin{align*}
		(\log\log N)^{-1/3+o(1)}.
	\end{align*}
	Then there exists a non-trivial three-term arithmetic progression in $A$.
\end{theorem}
Naslund~\cite{Nasl:rothprimes} further
improved the lowest admissible density to
$(\log\log N)^{-1+o(1)}$.
It should be noted that these transference arguments preserve, 
up to a logarithm, the exponent in the best known bounds 
for Roth's theorem by Sanders~\cite{S:roth},
on which they rely:
indeed Sanders established that
three-term arithmetic progressions may be found
in any subset of $[N]$ of density
at least $(\log N)^{-1+o(1)}$.

In the context of counting linear
patterns in primes~\cite{GT:lineqs}, 
Green and Tao introduced the notion of 
\textit{Cauchy-Schwarz complexity}\footnote{
A more subtle notion of complexity, 
called \textit{true complexity}, was later developed by 
Gowers and Wolf~\cite{GW:complexity}.
However it does not seem, at present, to cover the setting 
of unbounded prime-counting functions.
}
(abbreviated as complexity in the following)
for systems of integer linear forms.
Precisely, we say that a system of $t$ distinct linear forms
$(\psi_1,\dots,\psi_t)$ has complexity at most $s$ when, 
for every $i \in [t]$, it is possible to partition
the set of forms $\{ \psi_j, j \neq i \}$ into at most
$s+1$ sets, such that $\psi_i$ does not
belong to the linear span of any of those sets.
The condition of finite complexity is then
equivalent to requiring that no two forms
of the system be linearly dependent.
By extension, we define the complexity of a matrix $V$
to be that of any parametrization 
$\psi : \Z^d \twoheadrightarrow \Z^t \cap \Ker(V)$,
this property being independent of the choice of $\psi$. 

Systems of complexity at most one
may be analyzed by methods of classical Fourier analysis, 
whereas cases of higher complexities
require much more involved techniques~\cite{Go:sz,GT:U3inv}.
We focus on the case of complexity one here,
for it is possible to derive strong quantitative bounds
in that setting, and for it may provide insight on how
to quantify results of higher complexity.
On the qualitative side, it is known that 
a translation-invariant system
of equations $V \mathbf{y} = 0$ of finite complexity
is non-trivially solvable 
in any subset of the primes
of positive upper density: this follows
from the Green-Tao theorem~\cite{GT:apsinprimes}
on arithmetic progressions in the primes,
by a simple folklore argument\footnote{
Given a system $\psi : \Z^d \rightarrow \Ker(V) \cap \Z^t$
of finite complexity, pick $u \in \Z$ so that
all the values $c_i = \psi_i(u)$ are distinct.
A simple variation of the proof of Green and Tao~\cite{GT:apsinprimes}
makes it possible to find a pattern $(x + c_1 d,\dots,x + c_t d)$
with distinct coordinates in the primes, 
and this yields a non-trivial solution
of $V \mathbf{y} = 0$ since $x + c_i d = \psi_i(x+du)$.
}.
Our main finding is that, in the case of complexity one,
quantitative bounds of the quality of Helfgott and de Roton's
may be achieved.

\begin{theorem}
\label{thm:intro:primescplx1}
	Let $V \in \M_{r \times t}(\Z)$ be a
	translation-invariant matrix of rank $r$ 
	and complexity one.
	There exists a positive constant $C$ depending
	at most on $r,t,V$ such that,
	if $A$ is a subset of $\mathcal{P}_N$
	of density at least 
	\begin{align*}
		C (\log\log N)^{-1/25t},
	\end{align*}
	there exists $\by \in A^t$ with distinct coordinates
	such that $V \by = 0$.
\end{theorem}

Our argument also preserves the aforementioned feature
of Naslund's refinement of the Helfgott-de Roton
transference principle:
in the complexity one regime,
it converts logarithmic density bounds $(\log N)^{-\gamma}$
for Szemerédi-type theorems in the integers, 
to doubly logarithmic bounds $(\log\log N)^{-\gamma+\eps}$
for Szemerédi-type theorems in the primes.
We mention however that Theorem~\ref{thm:intro:primescplx1} 
is surpassed, in certain special cases, by results in the integers.
Indeed, an important result
of Schoen and Shkredov~\cite{SS:rothmany} states that
any single translation-invariant equation
in a least $6$ variables is non-trivially solvable
in any subset of $[N]$ of density $e^{-(\log N)^{1/6 - o(1)}}$,
and hence in $\mathcal{P}_N$, however it is not clear
whether or how that result 
extends to the case of several equations.
Furthermore, in certain ``degenerate'' cases 
where the $r \times t$ matrix $V$ 
may be subdivided into translation-invariant 
$r \times t_i$ submatrices, the system of equations 
may even be solvable at densities $N^{-c}$:
we refer to the work of Shapira~\cite{Shap:systeqs},
generalizing that of Ruzsa~\cite{Ruzsa:lineqs}, 
for precise statements.

To motivate Theorem~\ref{thm:intro:primescplx1}, 
we now give some illustrative examples
of systems of complexity one.
First, any single translation-invariant equation has
complexity one, although in that case
a simple modification of the argument
of Helfgott and de Roton~\cite{HR:rothprimes} 
yields Theorem~\ref{thm:intro:primescplx1}.
A more representative example of a system of complexity one is
that of ``$d$ points and their midpoints'',
corresponding to the set of equations
$(y_{ii} + y_{jj} = 2y_{ij})_{1 \leq i < j \leq d}$,
whose solutions over $\Q$ are parametrized, 
with some multiplicity, by\footnote{
This system is the linear part
of Example~4 from \cite[Section~1]{GT:lineqs},
composed with a certain surjection.
} $\psi(x)=(x_0 + x_i + x_j)_{1 \leq i \leq j \leq d}$.
It can be arduous in general to determine
whether a system of equations has complexity one:
Vinuesa~\cite{Vinuesa:magic} has 
determined, by an elaborate combinatorial argument,
that the system of translation-invariant equations
corresponding to magic $n \times n$ squares
has complexity one for $n \geq 4$.
For a more general discussion of the complexity one setting,
we refer to Section~\ref{sec:linalg}.

Next, we discuss the principal ideas behind
the proof of Theorem~\ref{thm:intro:primescplx1}.
The main structure of our argument follows
the transference principle, introduced by Green~\cite{G:rothprimes}
and further developped by Green and Tao~\cite{GT:apsinprimes}, 
and by which one lifts a dense subset of the primes to 
a dense subset of the integers.
More precisely, we initially follow the efficient transference
strategy of Helfgott and de Roton~\cite{HR:rothprimes},
which builds on that of Green and Tao~\cite{GT:restriction},
and we incorporate Naslund's~\cite{Nasl:rothprimes} estimates.
Denoting by $\lambda_A$ the renormalized indicator
function of a dense subset $A$ of the primes,
we therefore compare the average of $\lambda_A$ over 
$\psi$-patterns to that of 
a smoothed version $\lambda_A'$ of itself,
which behaves as a dense subset of the integers
of almost the same density.
As usual, there is a little technical subtelty in the form
of the $W$-trick, by which we
consider, instead of the set $A$,
its intersection with an arithmetic progression of modulus 
$W = \prod_{p \leq \omega} p$.
A critical feature of Helfgott and de Roton's argument~\cite{HR:rothprimes}
is then that it requires a modulus $\omega \sim c\log N$.

At this point we invoke a beautiful recent result of Shao~\cite{Shao:configs},
who improved on a first result of Dousse~\cite{Dousse:configs}, 
and generalized the logarithmic bounds of Bourgain~\cite{B:roth}
for Roth's theorem to a model system of complexity one.
More precisely, Shao~\cite{Shao:configs} investigated
the system $\psi(x)=(x_0 + x_i + x_j)_{1 \leq i \leq j \leq d}$,
and proved that a set $A$ of density $(\log N)^{-1/6d(d+1)+o(1)}$
in $[N]$ contains a non-trivial configuration $\psi(x) \in A^{d(d+1)/2}$.
As envisioned by Shao~\cite[p.~2]{Shao:configs},
his argument naturally extends to general systems of complexity one,
at the cost of adressing certain technical complications.
The first, and simplest step of our proof is therefore to formally
derive this extension, while also keeping track of the number
of pattern occurences.
Considering $\lambda'_A$ as a dense set of integers,
this extension then shows that $\lambda'_A$ has a large
pattern count.

Provided that we could prove that the difference of
pattern counts for $\lambda_A$
and $\lambda'_A$ is small, this would be enough
to conclude that the original set $A$
contains many $\psi$-configurations.
However, while the count of three-term progressions 
investigated by Helfgott and de Roton~\cite{HR:rothprimes} 
has a simple Fourier expression,
which can be controlled by restriction estimates for 
primes~\cite{GT:restriction}, 
such is not the case in general for systems of complexity one.
To address this issue, 
we bound the difference of pattern counts
via the generalized Von Neumann theorem
of Green and Tao~\cite{GT:lineqs}, which
in the complexity-one setting asserts that, 
given functions $f_1,\dots,f_t$ on $\Z_{N'}$ 
with $N' \sim CN$ majorized by a pseudorandom weight 
(a notion whose meaning shall be clear shortly), we have
\begin{align}
\label{eq:intro:genvonnm}
	\big| \E_{n \in \Z_{N'}^d} f_1 (\psi_1(n)) \dots f_t(\psi_t (n)) \big|
	\leq \| f_i \|_{U^2} + o(1)
\end{align}
as $N \rightarrow \infty$.
Properly quantified, the method of 
Green and Tao~\cite{GT:apsinprimes,GT:lineqs}
produces a $o(1)$ term of size $(\log N)^{-c}$ in the above,
however it requires a small modulus $\omega \sim c\log\log N$,
which is too expensive to apply the efficient 
transference estimates of Helgott and de Roton~\cite{HR:rothprimes}.

To majorize prime-counting functions
associated to $W$-tricked primes, 
Green and Tao use a weight $\nu : \Z_M \rightarrow \R^+$ 
constructed from a smoothly truncated convolution of the Möbius function,
whose averages where first considered by Goldston, Pintz and Yildirim~\cite{GPY:smallgaps}.
The $o(1)$-term arising in~\eqref{eq:intro:genvonnm} then depends
on the level of pseudorandomness of this weight,
and the key estimate we establish towards this is the asymptotic
\begin{align*}
	\E_{n \in \Z_{N'}^d} \nu(\theta_1(n)) \dots \nu(\theta_t(n))
	= 1 + O_{d,t,\theta}\bigg( \frac{1}{(\log N)^{1-o(1)}} \bigg),
\end{align*}
valid for every affine system 
$\theta : \Z_{N'}^d \rightarrow \Z_{N'}^t$
of finite complexity and bounded linear part,
and for a large modulus $\omega \sim c \log N$.
This corresponds to the ``linear forms condition'' 
in~\cite{GT:apsinprimes,GT:lineqs},
while we do not need the harder-to-quantify 
``correlation condition'' 
from there in our simpler setting.
Equipped with this estimate, we verify that
the functions $\lambda_A$ and $\lambda'_A$ 
used by Helfgott and de Roton are
majorized by averaged variants of $\nu$,
and we finally apply~\eqref{eq:intro:genvonnm} to bound
the difference of pattern counts.

\textbf{Remarks.}
Very recently, and while we were writing this article, 
Conlon, Fox and Zhao have completed an exposition of
the Green-Tao theorem~\cite{CFZ:GTexpo}, in which
they also revisited Green and Tao's computations on
correlations of GPY weights under
the assumption of finite complexity.
Their number-theoretic computations~\cite[Section~9]{CFZ:GTexpo} 
turn out to be quite similar to ours from Section~\ref{sec:gpy},
although our argument optimizes certain parameters further.

\textbf{Acknowledgements.}
We are grateful to our adviser Régis de la Bretèche
for valuable advice on writing.
We also wish to thank our friends
Crystel Bujold, Dimitri Dias, 
Oleksiy Klurman, Marzieh Mehdizad
for helpful discussions on many topics
of number theory.
We would further like to thank Pablo Candela,
Harald Helfgott, Neil Lyall, Eric Naslund, Hans Parshall
and Fernando Shao for interesting discussions on
problems related to this paper.

\textbf{Funding.}
This work was partially supported by
the ANR Caesar {ANR-12-BS01-0011}.

\section{Overview}
\label{sec:overv}

In this section we explain
the organization of this paper,
and how it relates to
the structure of our argument
presented in the introduction.

The preliminaries to our argument
are contained in Sections~\ref{sec:not} and~\ref{sec:linalg}.
The little notation we need is
introduced in Section~\ref{sec:not},
while Section~\ref{sec:linalg} is there
to gather (almost) all arguments 
of a linear algebraic nature
needed in the article.

With these prequisites in place,
the first logical part of our argument
is the aforementioned extension of
Shao's~\cite{Shao:configs} result,
and since it require few new ideas
we place it at the end of the article in Appendix~\ref{sec:apx:intg}.
The bulk of our proof of Theorem~\ref{thm:intro:primescplx1}
is then contained in Sections~\ref{sec:gpy}--\ref{sec:tsf}.
In Section~\ref{sec:gpy}, we carry out
the computation of correlations of the GPY weights
\begin{align*}
	\Lambda_{\chi,R,W}(n) = \Big( \frac{\phi(W)}{W} \log R \Big) 
	\bigg( \sum_{d | Wn + b} \mu(d) \chi\Big( \frac{\log d}{\log R} \Big) \bigg)^2,
\end{align*}
where $W = \prod_{p \leq \omega} p$
and $\chi$ is a certain smooth cutoff function.
We follow Green and Tao's original 
computation~\cite[Appendix~D]{GT:lineqs}, 
but we analyze the local Eulor factors involved
in more detail, in order 
to allow for a large modulus $\omega = c\log N$.
In Section~\ref{sec:rd}, we construct
a pseudorandom weight on $\nu$ over $\Z_M$
out of $\Lambda_{\chi,R,W} : \Z \rightarrow \R^+$ for a
larger scale $M \sim CN$,
taking care to preserve quantitative error terms.
We also state a quantitative version of
Green and Tao's generalized 
Von Neumann theorem~\cite[Appendix~C]{GT:lineqs}.
In Section~\ref{sec:tsf}, we prove Theorem~\ref{thm:intro:primescplx1},
by first lifting the problem to the integers
via the transference principle of 
Helfgott-de Roton~\cite{HR:rothprimes}
and the quantitative generalized Von Neumann theorem obtained earlier,
and by then applying the extension of Shao's result
derived in Appendix~\ref{sec:apx:intg}.

\section{Notation}
\label{sec:not}

We have attempted to respect most current conventions
of notation in additive combinatorics~\cite{AG:acbook} throughout, 
and therefore we keep this section to the bare minimum.

Given an integer $N$, we write $[N] = \{1,\dots,N\}$.
Given reals $x < y$, we also write  
$[x,y]_\Z = \Z \cap [x,y]$,
and we let $\mathcal{P}$ denote the set of all primes.
Given a property $\mathbf{P}$, we write 
$1( \mathbf{P} )$ for the boolean which equals $1$
when $\mathbf{P}$ is true, and $0$ otherwise.
When $X$ is a set and $\mathbf{P}_x$ is a property
depending on a variable $x \in X$, we write
\begin{align*}
	\p_{x \in X} ( \mathbf{P}_x ) 
	= |X|^{-1} \#\{ x \in X : \mathbf{P}_x \}.
\end{align*}
Given a function $f$ on $X$, we also write 
$\E_X f = \E_{x \in X} f(x) = |X|^{-1} \sum_{x \in X} f(x)$,
or simply $\E f$ when the set of averaging is clear from the context.

We make occasional use of Landau's $o$, $O$-notation and
of Vinogradov's asymptotic notations $f \ll g$, $f \gg g$, $f \asymp g$.
As is common in additive combinatorics, we also
let $c$ and $C$ denote positive constants whose value
may change at each occurence, and which are typically 
taken to be respectively very small or very large.
Unless otherwise stated, 
all implicit and explicit constants we introduce
are absolute: they do not depend on surrounding parameters.

Finally, we use several local conventions on notation,
and therefore we advise the reader to pay close
attention to the preamble of each section.

\section{Linear algebra preliminaries}
\label{sec:linalg}

In this section, we discuss the notion
of complexity of systems of linear forms,
following the very transparent exposition 
by Green and Tao in~\cite[Sections~1 and~4]{GT:lineqs},
and by Tao in~\cite{T:blognormalform}.
We also consider the simple problems of parametrizing
the kernel of a matrix corresponding to a system
of equations, and of defining an analog notion
of complexity for such a matrix.

We consider an integral domain $\bbA$,
together with its field of fractions $\bbK$;
in our article we only ever consider $\bbA = \Z$
or $\bbA = \Z_M$ with $M$ prime.
A linear form over the free module $\bbA^d$
naturally induces one over $\bbK^d$,
and accordingly all the linear algebra notions
are considered over $\bbK$.
This is somewhat overly formal, however it allows us to 
define certain notions for linear forms 
over $\Z$ and $\Z_M$ at once.
Note that throughout this article,
we consider systems of linear forms
$\psi : \bbA^d \rightarrow \bbA^t$
as formal triples $(\psi,d,t)$
to avoid repeatedly introducing
dimension parameters $d,t$.

\begin{definition}[Complexity]
\label{thm:linalg:cplxdef}
	Consider a system of linear forms
	$\psi = (\psi_1,\dots,\psi_t) : {\bbA^d \rightarrow \bbA^t}$.
	For $i \in [t]$, the complexity of $\psi$ at $i$
	is the minimal integer $s \geq 0$ 
	for which there exists a partition 
	$[t] \smallsetminus \{i\} = X_1 \sqcup \dots \sqcup X_{s+1}$ 
	into non-empty sets such that
	$\psi_i \notin \langle \psi_j : j \in X_k \rangle$ 
	for all $k \in [s+1]$, when such an integer 
	exists\footnote{
	In the special (and unimportant) case where $t=1$, we set the complexity at $i=1$ to $0$.
	}.
	Otherwise we set the complexity at $i$ to $\infty$.
	The complexity of $\psi$ is the maximum
	of the complexities of $\psi$ at $i$
	over all $i \in [t]$.
\end{definition}

We also recall the following important 
observation from~\cite[Section~1]{GT:lineqs}.

\begin{lemma}
\label{thm:linalg:fincplx}
	A system of linear forms 
	$\psi = (\psi_1,\dots,\psi_t) : \bbA^d \rightarrow \bbA^t$
	has finite complexity if and only if
	no two forms $\psi_i,\psi_j$ with $i \neq j$ are linearly dependent.
\end{lemma}

We next recall the standard notion of normal form,
and to do so we introduce a slightly 
non-standard piece of terminology.
We say that a linear form 
$\theta(x_1,\dots,x_d) = a_1 x_1 + \dots + a_d x_d$
\textit{depends} on the variable $x_k$
when $a_k \neq 0$;
we do not mean this in an exclusive sense
so that the form may also depend on other variables.
While that definition may seem mathematically akward,
it corresponds to the intuitive way to
think about explicit system of forms.

\begin{definition}[Normal form]
\label{thm:linalg:normaldef}
	A system of linear forms 
	$\psi = (\psi_1,\dots,\psi_t) : \bbA^d \rightarrow \bbA^t$
	is in \emph{exact} $s$-normal form at $i \in [t]$ when
	there exists a set of indices
	$J_i \subset [d]$ such that $|J_i| = s + 1$ and
	\begin{enumerate}
		\item $\psi_i(x_1,\dots,x_d)$ depends on all variables $x_k, k \in J_i$,
		\item for all $j \neq i$, $\psi_j(x_1,\dots,x_d)$ does not depend on all variables $x_k, k \in J_i$.
	\end{enumerate}
	We say that $\psi$ is in $s$-normal form when
	it is in exact $s_i$-normal form with $s_i \leq s$ at every $i \in [t]$.
\end{definition}

As explained in~\cite[Section~4]{GT:lineqs}, a system $\psi$
in exact $s$-normal form at $i$ has complexity at most $s$ at $i$,
and conversely one may always put a system of complexity~$s$
in $s$-normal form, up to adding a certain number of ``dummy'' variables.

\begin{proposition}[Normal extension]
\label{thm:linalg:normalextension}
	A system of linear forms 
	$\psi : \Z^d \rightarrow \Z^t$ of complexity $s$
	admits an $s$-normal extension 
	$\psi' : \Z^{d+e} \rightarrow \Z^t$
	of the form $\psi'(x,y) = \psi(x + \varphi(y))$, where 
	$\varphi : \Z^e \rightarrow \Z^d$ is a linear form.
\end{proposition}

We will also have the occasion to consider systems 
of affine-linear forms,
often abbreviated as ``affine systems'' 
throughout the article. 
Consistently with~\cite{GT:lineqs}, 
we write an affine system $\psi$
as $\psi = \psi(0) + \dot{\psi}$, 
where $\dot{\psi}$ is the linear part of $\psi$, 
and we extend previous definitions 
by declaring $\psi$ to be of complexity $s$ 
or in $s$-normal form when its linear part is.
We also need to consider reductions
of forms modulo a large prime $M$ later on,
in which case we need to keep track of the size
of the coefficients of the forms involved.

\begin{definition}[Form and matrix norms]
\label{thm:linalg:norms}
	Suppose that $\psi = (\psi_1,\dots,\psi_t) : \bbA^d \rightarrow \bbA^t$
	is an affine system, and write
	$\psi_i(x_1,\dots,x_d) = a_{i1} x_1 + \dots + a_{id} x_d + b_i$ 
	for every $i \in [t]$.
	When $\bbA = \Z$ and $M \geq 1$, we define
	\begin{align*}
		\| \psi \|_{M} = \sum_{i \in [t]} \sum_{j \in [d]} |a_{ij}| 
		+ \sum_{i \in [t]} (|b_i|/M),
	\end{align*}
	and we simply write $\| \psi \|$ when all $b_i$ are zero.
	When $\bbA = \Z_M$, we define
	\begin{align*}
		\| \psi \| = \sum_{i \in [t]} \sum_{j \in [d]} \| a_{ij} \|_{\T_M}
		+ \sum_{i \in [t]} \| b_i / M \|_{\T}
	\end{align*}
	where $\| \cdot \|_{\T_L} = d( \cdot, L\Z )$.
	Finally, for 
	$V = [\lambda_{ij}] \in \M_{r \times t}(\Z)$,
	we write $\| V \| = \sum_{i,j} |\lambda_{ij}|$ .
\end{definition}

We now return to our main topic of interest,
that is, translation-invariant equations in the integers.
As for systems of forms, we consider matrices
$V \in \M_{r \times t}(\Z)$ as formal triples $(V,r,t)$.

\begin{definition}
\label{thm:linalg:tslinvmatrix}
	We say that $V = [a_{ij}] \in \M_{r \times t}(\Z)$
	is translation-invariant when
	\begin{align*}
		a_{i1} + \dots + a_{it} = 0 \quad \forall i \in [r].
	\end{align*}
\end{definition}

Given a matrix $V \in \M_{r \times t}(\Z)$ 
corresponding to a system
of equations $V \by = 0$, we now define
the complexity of $V$ at an indice $i \in [t]$,
and its global complexity, to be that of any system of linear forms 
$\psi : \Q^d \twoheadrightarrow \Ker(V)$.
The following proposition ensures that
such a definition does not depend on the 
choice of parametrization $\psi$.

\begin{proposition}[Matrix complexity criterion]
\label{thm:linalg:matrixcplx}
	Consider a matrix $V \in \M_{r \times t}(\Z)$
	with lines $L_1,\dots,L_r$ and $t \geq 2$,
	and a system of linear forms
	$\psi : \Q^d \twoheadrightarrow \Ker(V)$.
	Then $\psi$ has complexity at most $s_0$ at $i \in [t]$
	if and only if there exists $0 \leq s \leq s_0$ and a partition
	$[t] \smallsetminus \{i\} = X_1 \sqcup \dots \sqcup X_{s+1}$ 
	into non-empty sets such that, for every $k \in [s+1]$,
	\begin{align*}
		\big( e_i + \textstyle\sum_{j \in X_k} \Q e_j \big) \cap
		\langle \transp{L_1} , \dots , \transp{L_r} \rangle
		= \varnothing,
	\end{align*}
	where $(e_i)_{1 \leq i \leq t}$ is the canonical basis of $\Q^t$.
\end{proposition}

\begin{proof}
Consider $i \in [t]$ and a partition
$[t] \smallsetminus \{i\} = X_1 \sqcup \dots X_{s+1}$
into non-empty sets.
For any $k \in [s+1]$ and
$\lambda \in \Q^{X_k}$,
we have an equivalence
\begin{align*}
	&\psi_i + \textstyle\sum_{j \in X_k} \lambda_j \psi_j = 0 
	\\
	\Leftrightarrow\quad
	&x_i + \textstyle\sum_{j \in X_k} \lambda_j x_j = 0 
	\ \text{for all $x \in \Ker(V)$} \\
	\Leftrightarrow\quad
	&e_i + \textstyle\sum_{j \in X_k} \lambda_j e_j \in \Ker(V)^{\bot}.
\end{align*}
Furthermore, by orthogonality in $\Q^t$, 
\begin{align*}
	\Ker(V)^{\bot} 
	= \big( \langle \,\transp{L_1} , \dots , \transp{L_t} \rangle^{\bot} \big)^{\bot}
	= \langle \,\transp{L_1} , \dots , \transp{L_r} \rangle.
\end{align*}
Therefore $\psi_i \in \langle \psi_j , j \in X_k \rangle$
if and only if there exists $\lambda \in \Q^{X_k}$
such that $e_i + \sum_j \lambda_j e_j \in \langle \, \transp{L_1} , \dots , \transp{L_r} \rangle$.
The proposition follows by considering the contrapositive.
\end{proof}

We shall have the occasion to work 
with two standard types
of parametrizations for the integer kernel
of a translation-invariant matrix.
The first is the usual normal form,
which is useful when working with primes,
while the second has an added shift variable,
which is useful for the regularity computations 
of Appendix~\ref{sec:apx:intg}.
In both cases, it is critical to 
work with a base parametrization $\psi$ in normal form,
in order to bound averages over
patterns $(\psi_1(n),\dots,\psi_t(n))$
by a certain Gowers norm
(see Propositions~\ref{thm:rd:genvonnm}
and~\ref{thm:intg:largeTtolargeU2} below).

\begin{proposition}[Kernel parametrization]
\label{thm:linalg:kernelparam}
	Suppose that $V \in \M_{r \times t}(\Z)$
	is a translation-invariant matrix
	of rank $r$ and complexity at most $s$.
	Then there exists a linear surjection
	\begin{align*}
		\psi : \Z^d \twoheadrightarrow \Z^t \cap \Ker(V)
	\end{align*}
	in $s$-normal form.
	An alternate linear surjection is then given by
	\begin{align*}
		\varphi : \Z^{d+1} \twoheadrightarrow \Z^t \cap \Ker(V),
	\end{align*}
	where $\varphi$ is defined by
	$\varphi_i(x_0,x) = x_0 + \psi_i(x)$ for 
	every $i \in [t]$ and $(x_0,x) \in \Z \times \Z^d$.
\end{proposition}

\begin{proof}
The set $\Z^t \cap \Ker(V)$ is a lattice which is easily seen
to be of rank $t-r$ (e.g.~by first solving $V \by = 0$ over $\Q$,
then clearing denominators),
so that there exists a linear isomorphism
$\psi : \Z^{t-r} \isom \Z^t  \cap \Ker(V)$
of complexity at most $s$.
Since extensions in the sense of
Proposition~\ref{thm:linalg:normalextension} preserve the
image of a form, we may choose an alternate linear parametrization
$\psi' : \Z^d \isom \Z^t \cap \Ker(V)$ in $s$-normal form
for a certain $d \geq t-r$.

Since the matrix $V$ is translation-invariant, we have
$V \mathbf{1} = 0 $, where $\mathbf{1} = (1,\dots,1)$.
Therefore we may define another surjection
$\varphi: \Z \times \Z^d \twoheadrightarrow \Z^t \cap \Ker(V)$ by
$\varphi (x_0,x) = x_0 \mathbf{1} + \psi'(x)$.
\end{proof}

Note that a system of linear forms 
$\psi : \Z^d \rightarrow \Z^t$ 
in $1$-normal form is, at every position $i \in [t]$, 
either in exact {$0$-normal} form
or in exact {$1$-normal} form.
In practice we can always 
eliminate the first possibility,
and while not of fundamental importance,
this fact allows us to simplify our argument
in some places.

\begin{proposition}
\label{thm:linalg:exact1normalform}
	Suppose that $V \in \M_{r \times t}(\Z)$ 
	is a matrix of complexity one with no zero 
	columns and $t \geq 3$, 
	and $\psi : \Z^d \twoheadrightarrow \Z^t \cap \Ker(V)$
	is a system of linear forms
	in $1$-normal form.
	Then $\psi$ is in exact $1$-normal form
	at every $i \in [t]$.
\end{proposition}

\begin{proof}
Let $C_1,\dots,C_t$ denote the columns of $V$
and consider an indice $i \in [t]$.
By Proposition~\ref{thm:linalg:matrixcplx},
the statement that $\psi$ has non-zero complexity 
at $i$ is equivalent to
\begin{align*}
	\psi_i \in \langle \psi_j , j \neq i \rangle
	&\Leftrightarrow
	\big( e_i + \textstyle\sum_{j \neq i} \Q e_j \big) \cap 
	\langle \transp{L_1} , \dots , \transp{L_r} \rangle \neq \varnothing \\
	&\Leftrightarrow
	\exists \mu \in \Q^r : \textstyle\sum_{j = 1}^r \mu_j \transp{L_j} \cdot e_i = 1 \\
	&\Leftrightarrow
	\exists \mu \in \Q^r : \mu \cdot C_i \neq 0,
\end{align*}
and this last condition is satisfied if and only if $C_i$ is non-zero.
Since $\psi$ may have complexity only zero or one,
this concludes the proof under our assumption on the matrix.
\end{proof}

By similar orthogonality considerations, one can
establish that a matrix has complexity at most one
if and only if when any of its columns is excluded,
the set of remaining columns may be partitioned
into two classes, in a way that the excluded column
belongs to the linear span of each class.
This provides a concrete criterion,
which overlaps strongly with a set of conditions
designed by Roth~\cite{R:systeqs} and resurfacing in work of 
Liu, Spencer and Zhao~\cite{LSZ:rothfunction,LSZ:rothabelian},
but we do not dwelve on this relationship here.
One more simple fact we require about
(translation-invariant) systems of equations
is a bound on the number of integer solutions
with two equal coordinates in a box.

\begin{lemma}[Number of degenerate solutions]
\label{thm:linalg:trivsols}
	Suppose that $V \in \M_{r \times t}(\Z)$
	has rank $r$ and finite complexity, and let $i,j$
	be two distinct indices in $[t]$.
	Then 
	\begin{align*}
		\#\{ y \in [-N,N]_\Z^t : Vy = 0 \ \text{and}\ y_i = y_j \} \ll_V N^{t - r - 1}.
	\end{align*}
\end{lemma}

\begin{proof}
Consider the hyperplane
$H = \{ y \in \Q^t : y_i = y_j \}$.
The subspace $\Ker(V) \cap H$ of $\Q^t$
has dimension less than $t-r-1$,
since $\Ker(V)$ is not contained in $H$:
indeed if this were the case,
there would exist a parametrization
$\psi : \Z^d \twoheadrightarrow \Z^t \cap \Ker(V)$ 
with $\psi_i = \psi_j$,
contradicting the assumption of finite complexity.
The bound then follows by simple 
linear algebraic considerations. 
\end{proof}

Finally, we collect together some facts about
the preservation of certain properties of affine systems under the
operations of reduction modulo $M$ or lifting from $\Z_M$ to $\Z$.
We omit the proofs, which are accessible
by simple linear algebra.

\begin{fact}
\label{thm:linalg:reducenormal}
	Suppose that $V \in \M_{r \times t}(\Z)$ is a translation-invariant matrix
	of rank $r$ and $\psi : \Z^d \twoheadrightarrow \Z^t \cap \Ker_\Q(V)$ 
	is a system of linear forms in exact $s_i$-normal form over $\Z$
	at every $i \in [t]$.
	Provided that $M > \max( t! \| \psi \|^t, r! \|V\|^r)$,
	$\psi$ reduces modulo $M$ to 
	a system of linear forms 
	$\theta : \Z_M^d \twoheadrightarrow \Ker_{\Z_M} (V)$	
	is in exact $s_i$-normal form over $\Z_M$ at every $i \in [t]$,
	and such that $\| \theta \| = \| \psi \|$.
\end{fact}


\begin{fact}
\label{thm:linalg:liftfincplx}
	Suppose that $\theta : \Z_M^d \rightarrow \Z_M^t$ is
	an affine system of finite complexity over $\Z_M$,
	and $M > 2 \| \dot{\theta} \|$.
	Then $\theta$ is the reduction
	modulo $M$ of an affine system $\psi : \Z^d \rightarrow \Z^t$
	of finite complexity over $\Z$ and
	such that $\| \psi \|_M = \| \theta \|$,
	$\| \dot{\psi} \| = \| \dot{\theta} \|$.
\end{fact}

\section{Correlations of GPY weights}
\label{sec:gpy}

The aim of this section is to construct efficient
pseudorandom weights over $\Z$ majorizing the measure
associated to $W$-tricked primes.
The weight we consider (see Definition~\ref{thm:gpy:nugpydefinition} below)
is a truncated divisor sum whose correlations were
first investigated by Goldston, Pintz and Yildirim~\cite{GPY:correl,GPY:smallgaps}
in the context of small gaps between primes.
Green and Tao~\cite{GT:apsinprimes,GT:lineqs} investigated 
its pseudorandom behavior in greater generality,
and this weight is by now a standard tool,
e.g.~in the context of detecting polynomial patterns in 
primes~\cite{TZ:polynomial,TZ:polynomialerratum,LeWolf:polynomial}.
Throughout this section, we will assume familiarity
with~\cite[Appendix~D]{GT:lineqs}.

We now fix an integer $N$
larger than some absolute constant,
and we let $\omega \geq 1$ be a parameter.
We also let $W = \prod_{p \leq \omega} p$ and
we fix an integer $b$ such that $(b,W) = 1$.
It is useful to have a notation
for the normalized indicator function 
of $W$-tricked primes.

\begin{definition}[Measure of $W$-tricked primes]
\label{thm:gpy:wprimesmeasuredef}
	We let
	\begin{align*}
		\lambda_{b,W}(n)
		&= \frac{\phi(W)}{W} (\log N) \cdot 
		1( n \in [N] \ \text{and}\  b + Wn \in \mathcal{P} ).
	\end{align*}
\end{definition}

Our goal is thus to construct a
weight function over $\Z$ majorizing $\lambda_{b,W}$,
and satisfying strong pseudorandomness asymptotics.
Note that $o(1)$ terms throughout this article
are to be understood as $N \rightarrow \infty$,
and do not depend on any dimension 
or any affine system involved.

\begin{proposition}[Pseudorandom majorant over $\Z$]
\label{thm:gpy:quantpseudordoverZ}
	Let $D \geq 1$ be a parameter.
	There exists a
	constant $C_D$ such that the following holds. 
	For $N \geq C_D$ and $\omega = c_0 \log N$,
	there exists $\nu : \Z \rightarrow \R^+$
	such that, for every $\eps > 0$,
	\begin{align*}
		0 \leq \lambda_{b,W} \ll_D \nu \ll_\eps N^\eps
	\end{align*}		
	and, for any $P \geq N^{c_1}$ and
	any affine system $\psi : \Z^d \rightarrow \Z^t$ 
	of finite complexity and such that 
	$d,t,\| \dot{\psi} \| \leq D$, 
	\begin{align}
	\label{eq:gpy:pseudordoverZ}
		\E_{n \in [P]^d}
		\,\nu \big[ \psi_1(n) \big]
		\dots
		\nu \big[ \psi_t(n) \big] =
		1 + O_D \bigg( \frac{1}{(\log N)^{1-o(1)}} \bigg).
	\end{align}
\end{proposition}

Note that simply applying~\cite[Theorem~D.3]{GT:lineqs}
would be insufficient for our purpose, since the
error there is $e^{ O( \sqrt{\omega} ) } (\log N)^{-1/20}$
and therefore it is non-trivial only for 
$\omega \leq c(\log\log N)^{2}$,
thus rendering the methods of 
Helfgott and de Roton~\cite{HR:rothprimes} unapplicable.
The argument of~\cite{GT:apsinprimes} also requires
a modulus $\omega \leq c\log\log N$.
Our construction follows closely
that in~\cite[Appendix~D]{GT:lineqs},
however with one important difference:
we make a stronger assumption
of finite complexity on the system of linear forms,
and under this assumption we obtain
improved estimates on the Euler products involved.
We also remark that for the purpose of
proving Theorem~\ref{thm:intro:primescplx1},
any error term of the form $(\log N)^{-c}$
in~\eqref{eq:gpy:pseudordoverZ} would suffice,
however we take the opportunity here 
to determine the highest level of pseudorandomness 
attainable from Green and Tao's approach.

We let $R=N^\eta$, where $\eta$ is a small
positive constant specified later on. 
We consider a family of reals $\rho : \N \rightarrow [0,e^2]$
such that $\rho(1) = 1$ and with support on $[R]$, 
which we also specify later on.
Our main object of study in this section
is the following expression, which may be seen
as a smooth Selberg-type weight.

\begin{definition}[GPY weight]
\label{thm:gpy:nugpydefinition}
	We let $h_{R,W} = \frac{\phi(W)}{W} \log R$ and
	\begin{align*}
		\Lambda_{\chi,R,W}(n)
		= h_{R,W} \bigg( \sum_{m | Wn + b} \mu(m) \rho(m) \bigg)^2.
	\end{align*}
\end{definition}

The pseudorandom weight we seek will
turn out to be a scalar multiple
of the above function:
we defer the precise choice of normalization
until the end of the proof of 
Proposition~\ref{thm:gpy:quantpseudordoverZ}.

\begin{lemma}
\label{thm:gpy:gpymajorizes}
	When $\omega = c_0\log N$ and $R = N^{\eta}$ with 
	$0 < \eta \leq c_0/2$, we have 
	\begin{align*}
		0 \leq 	\lambda_{b,W} \ll_\eta \Lambda_{\chi,R,W} \ll_\eps N^{\eps}
	\end{align*}
	for every $\eps > 0$.
\end{lemma}

\begin{proof}
If $\lambda_{b,W}(n)$ is non-zero, 
$Wn + b$ is a prime of size at least 
$W > N^{c_0/2}$, for $N$ large enough.
Therefore any non-trivial divisor of $Wn+b$ is larger than $R$, 
so that 
$\Lambda_{\chi,R,W}(n) = \frac{\phi(W)}{W} (\log R) \rho(1) 
\leq \eta^{-1} \lambda_{b,W}(n)$.
The last inequality follows from
standard bounds on the 
divisor function~\cite{Ten:intro}.
\end{proof}

We now say more on the choice of weights $\rho(m)$.
We let 
\begin{align*}
	\rho(m) = \chi\Big( \frac{\log m}{\log R} \Big)
	\quad\text{where}\quad
	\chi(x) = 1_{[-1,1]}(x) \cdot e^{x+1} e^{-1/(1-x^2)}
\end{align*}
is the usual bump function multiplied by an exponential.
By Fourier inversion, we may write
$\chi(x) = \int_{-\infty}^{\infty} \varphi(\xi) e^{-(1+i\xi)x} \mathrm{d}\xi$
for every $x \in [-1,1]$,
where $\varphi$ is the Fourier transform
of $1_{[-1,1]}(x) e^{1-1/(1-x^2)}$,
and thus decays as\footnote{
Using a weaker decay $\ll (1+|\xi|)^{-A}$ instead 
would yield a slightly 
weaker error term $(\log N)^{-1+\eps}$
in Proposition~\ref{thm:gpy:quantpseudordoverZ}.}
$\varphi(\xi) \ll e^{-c|\xi|^{1/2}}$
(see e.g.~\cite{Johnson:saddle}).
The interest in this choice is that, by truncation
at a parameter $L \geq 1$, we may write
\begin{align}
\label{eq:gpy:weighttruncation}
	&\phantom{(m \leq R).} &
	\rho(m) &= \int_{-L}^{L} m^{-(1+i\xi)/\log R} \varphi(\xi) d\xi
	+ O\big( e^{-cL^{1/2}} \big)
	&&(m \leq R).
\end{align}
This has the effect of introducing a small negative power of $m$ 
in the Euler products arising in computations,
which simplifies their evaluation greatly.

We now begin the proof of Proposition~\ref{thm:gpy:quantpseudordoverZ}.
We fix $D \geq 1$ and $\omega = c_0 \log N$,
so that we may assume that $\omega$
is larger than any fixed constant depending on $D$.
We then consider a system of affine-linear forms
$\psi : \Z^d \rightarrow \Z^t$ of finite complexity
such that $d,t,\| \dot{\psi} \| \leq D$.
We let further implicit constants and explicit
unsuscripted constants $c,C$ depend on $d,t,\| \dot{\psi} \|$,
while subscripted constants $c_0,c_1,\dots$ are absolute.

The first step of the proof is to unfold divisor sums in the
correlation of divisor sums, and it is useful in 
this regard to introduce the notation 
$\Omega = [t] \times [2]$.
Note also that the prime in $\sideset{}{'}\sum$
means that the summation is restricted to square-free numbers.
The following constitutes the beginning of
the proof of~\cite[Theorem~D.3]{GT:lineqs},
which we do not reproduce.

\begin{proposition}[Unfolding sums]
\label{thm:gpy:divsumunfolding}
Given $(m_{ij}) \in \N^{\Omega}$, write $m_i = [m_{i1},m_{i2}]$ and
\begin{align*}
	\alpha(m_1,\dots,m_t) = \p_{n \in \Z_m^d}\big(\, m_i | W \psi_i(n) + b \quad \forall i \in [t] \,\big).
\end{align*}
Let also $P \geq 1$. Then
\begin{align*}
	&h_{R,W}^{-t} \sum_{n \in [P]^d}
	\Lambda_{\chi,R,W} \big[ \psi_1(n) \big] \dots \Lambda_{\chi,R,W} \big[ \psi_t(n) \big] \\
	= &P^d
	\cdot \sideset{}{'} \sum_{ (m_{ij}) \in \N^{\Omega} }
	\alpha(m_1,\dots,m_t)
	\prod_{(i,j) \in \Omega} \mu(m_{ij}) \rho(m_{ij})
	+ O( R^{2|\Omega|} P^{d-1} )
\end{align*}
\end{proposition}

Before proceeding further, we
analyze the function $\alpha$
appearing in Proposition~\ref{thm:gpy:divsumunfolding}.
By the Chinese Remainder theorem,
$\alpha(m_1,\dots,m_t)$ is multiplicative
in the variables $m_{ij}$, keeping
in mind that $m_i = [m_{i1},m_{i2}]$.
Writing $m_{ij} = p^{r_{ij}}$,
$r_i = \max (r_{i1},r_{i2})$,
and $B = \{ (i,j) \in \Omega : r_{ij} = 1 \}$,
we have $r_i = 1$ if and only if $r_{ij} =1 $
for some $j \in [2]$, that is, if and only if the slice $B_i$
of $B$ at $i$ is non-empty.
Therefore
\begin{align}
\label{eq:gpy:alphasetdef}
	\alpha(p^{r_1},\dots,p^{r_t})
	= \p_{n \in \Z_p^d} \big( p | W \psi_i(n) + b \quad \forall i : B_i \neq \varnothing \big) 
	\eqqcolon \alpha(p,B).
\end{align}
Motivated by this, we say that a non-empty set $B \subset \Omega$
is \textit{vertical} when, for some $i \in [t]$, 
we have $B \subset \{i\} \times [2]$.
We now estimate the size of the factors $\alpha(p,B)$.

\begin{proposition}[Local probabilities]
\label{thm:gpy:localprobs}
For $B \neq \varnothing$,
we have
\begin{align*}
	\alpha(p,B) =
	\begin{cases}
		0 & \text{if $p \leq \omega$} \\
		p^{-1} &\text{if $p > \omega$ and $B$ is vertical} \\
		O(p^{-2}) & \text{if $p > \omega$ and $B$ is not vertical}
	\end{cases}
\end{align*}	
\end{proposition}

\begin{proof}
Recall that $\alpha(p,B)$ is defined by~\eqref{eq:gpy:alphasetdef}.
When $p \leq \omega$,
we have $p | W$ and $(b,W)=1$,
therefore $p$ does not divide any value $W \psi_i(n) + b$ 
and $\alpha(p,B) = 0$.
When $p > \omega > \| \dot{\psi} \|$, we have $p \nmid W$
and $W \dot{\psi}_i \neq 0$ in $\Z_p$ for every $i \in [t]$.
When $B$ is vertical, there is only one $i$ such that
$B_i$ is non-empty and therefore
$\alpha(p,B) = p^{-1}$.
When $B$ is not vertical, there are at least
two indices $i,j$ such that $B_i,B_j \neq \varnothing$.
Since $p > \omega > 2\| \dot{\psi} \|^2$, the linear forms
$\dot{\psi}_i$ and $\dot{\psi}_j$ 
are linearly independent over $\Z_p$,
and therefore $\alpha(p,B) \leq p^{-2}$.
\end{proof}

For reasons that shall be clear in a moment,
we define the following Euler factor.
\begin{definition}[Euler factor]
\label{thm:gpy:eulerfactordef}
	Let $\xi \in \R^{\Omega}$ and $z_{ij}=(1+i\xi_{ij})/\log R$.
	We let
	\begin{align*}
		E_{p,\xi} =
		\sum_{B \subset \Omega} (-1)^{|B|} \alpha(p,B) p^{-\sum_{(i,j) \in B} z_{ij}}. 
	\end{align*}
\end{definition}
The local estimates of Proposition~\ref{thm:gpy:localprobs}
and the fact that $\re(z_{ij}) > 0$
ensure the absolute convergence of the product $\prod_p E_{p,\xi}$.
We now return to the unfolded sum in 
Proposition~\ref{thm:gpy:divsumunfolding},
in which we proceed to replace the weights $\rho(m)$
by their truncated Fourier expression~\eqref{eq:gpy:weighttruncation}.
This step being again well described in~\cite[Appendix~D]{GT:lineqs},
we do not include the proof here.

\begin{proposition}[Unfolding integrals]
\label{thm:gpy:integralsunfolding}
Writing $m_i = [m_{i1},m_{i2}]$, we have, for any $L \geq 1$,
\begin{align*}
	&\sideset{}{'}\sum_{(m_{ij}) \in \N^{\Omega}} 
	\alpha(m_1,\dots,m_t)	
	\prod_{(i,j) \in \Omega}
	\mu(m_{ij}) \rho(m_{ij}) \\
	= &\idotsint\limits_{[-L,L]^{\Omega}}
	\prod_p E_{p,\xi} 
	\prod_{(i,j) \in \Omega} 
	\varphi(\xi_{ij}) \mathrm{d}\xi_{ij}
	+O\big( e^{-cL^{1/2}} (\log R)^{|\Omega|} \big).
\end{align*}
\end{proposition}

With the estimates on local probabilities at hand, 
one can easily estimate the Euler product arising in 
Proposition~\ref{thm:gpy:integralsunfolding}.

\begin{proposition}[Euler product estimate]
\label{thm:gpy:eulerproductestimate}
	Let $1 \leq L \leq \tfrac{c\log R}{\log\omega}$ be a parameter.
	For every $\xi \in [-L,L]^{\Omega}$, we have
	\begin{align*}
		\prod_p E_{p,\xi} =
		\bigg( 1 + O\Big( \frac{1}{\omega} + \frac{L \log\omega}{\log R} \Big) \bigg)
		\cdot h_{R,W}^{-t} \cdot
		\prod_{\textrm{$B$ vertical}} \bigg( \sum_{(i,j) \in B} (1 + i\xi_{ij}) \bigg)^{-(-1)^{|B|}} .	
	\end{align*}
\end{proposition}

\begin{proof}
Note at the outset the useful identity
$\sum_{\textrm{$B$ vertical}} (-1)^{|B|} = -t$,
and write $z_{ij} = (1+i\xi_{ij})/\log R$
as in Definition~\ref{thm:gpy:eulerfactordef}.
By Proposition~\ref{thm:gpy:localprobs}, we have
\begin{align*}
	\prod_p E_{p,\xi}
	&= \prod_{p > \omega} \bigg( 1 + \sum_{\textrm{$B$ vertical}} (-1)^{|B|} p^{-1-\sum_B z_{ij}} + O( p^{-2} ) \bigg) \\
	&=  ( 1 + O( \omega^{-1} ) )
		\prod_{p > \omega} \prod_{\textrm{$B$ vertical}} ( 1 - p^{-1-\sum_B z_{ij}} )^{-(-1)^{|B|}}
\end{align*}
Since $p^{-z} = 1 + O(\frac{L\log p}{\log R})$
for $p \leq \omega$ and $|z| \leq L/\log R$, we have further
\begin{align*}	
	\prod_p E_{p,\xi}
	= 	( 1 + O( \omega^{-1} ) )
	\prod_{p \leq \omega} \Big( 1 - p^{-1} + O\Big( \frac{L\log p}{p\log R} \Big) \Big)^{-t}
	\prod_{\textrm{$B$ vertical}} \zeta( 1 + \textstyle\sum_B z_{ij} )^{(-1)^{|B|}}.
\end{align*}
Using the fact that $\zeta(s) = \frac{1}{s-1}( 1 + O(|s-1|) )$ 
for $\re(s) > 1$, it follows that
\begin{align*}
	\prod_p E_{p,\xi} &= 
	\Big( 1 + O\Big( \frac{1}{\omega} + \frac{L\log \omega}{\log R} + \frac{L}{\log R} \Big) \Big)
	\Big( \frac{\phi(W)}{W}	\Big)^{-t}
	\prod_{\textrm{$B$ vertical}} \Big( \sum_B z_{ij} \Big)^{-(-1)^{|B|}} ,
\end{align*}
which concludes the proof upon recalling that $z_{ij} = (1+i\xi_{ij})/\log R$.
\end{proof}

At this stage, the following sieve factors arise.

\begin{definition}[Sieve factor]
	We let
	\begin{align*}
		c_{\chi,2} = \iint_{\R^2} \frac{(1+i\xi)(1+i\xi')}{2 + i(\xi + \xi')} 
		\varphi(\xi)\varphi(\xi')\mathrm{d}\xi\mathrm{d}\xi'.
	\end{align*}
\end{definition}

The last step is to replace the euler product 
$\prod_p E_{p,\xi}$ in Proposition~\ref{thm:gpy:integralsunfolding}
by its approximation obtained in Proposition~\ref{thm:gpy:eulerproductestimate},
and to extend the range of integration back to $\R$.
Again, we refer to~\cite[Appendix~D]{GT:lineqs} for 
the proof of this familiar step.

\begin{proposition}[Refolding integrals]
\label{thm:gpy:integralsrefolding}
	Provided that $1 \leq L \leq \tfrac{c\log R}{\log\omega}$, we have
	\begin{align}
		\label{eq:gpy:integralsrefolding}
		h_{R,W}^t \idotsint\limits_{[-L,L]^{\Omega}} \prod_p E_{p,\xi} 
		\prod_{(i,j) \in \Omega} \varphi( \xi_{ij} ) \mathrm{d}\xi_{ij}
		= c_{\chi,2}^t + O\bigg( e^{-cL^{1/2}} + \frac{1}{\omega} 
		+ \frac{L\log\omega}{\log R}  \bigg).
	\end{align}
\end{proposition}


We also quote~\cite[Lemma~D.2]{GT:lineqs},
which provides an explicit formula for $c_{\chi,2}$.

\begin{lemma}
\label{thm:gpy:sievefactor}
	We have $c_{\chi,2} = \int_0^\infty |\chi'(x)|^2 \mathrm{d}x$.
\end{lemma}

We may now combine the previous successive approximations to the original sum
and optimize the parameter $L$ to obtain 
Proposition~\ref{thm:gpy:quantpseudordoverZ}.

\medskip

\textit{Proof of Proposition~\ref{thm:gpy:quantpseudordoverZ}.}
Let $P \geq 1$. 
Combining Propositions~\ref{thm:gpy:divsumunfolding},
\ref{thm:gpy:integralsunfolding} and~\ref{thm:gpy:integralsrefolding},
we see that the average 
$\E_{n \in [P]^d} \prod_{i \in [t]} 
\Lambda_{\chi,R,W} \big[ \psi_i(n) \big]$
is equal to
\begin{align*}
	c_{\chi,2}^t + 
	O\bigg( e^{-cL^{1/2}} (\log R)^{O(1)}
	+ \frac{1}{\omega} 
	+ \frac{L\log\omega}{\log R}
	+ \frac{R^{5t}}{P} \bigg),
\end{align*}
provided that $L \leq \frac{c\log R}{\log\omega}$.
Recall now that $\omega = c_0 \log N$.
Assuming that $P \geq N^{c_1}$,
we choose $L = C(\log\log N)^2$ and 
$R = N^{c_2/t}$ for a small $c_2 > 0$, so that
\begin{align}
\label{eq:gpy:prefinalpseudord}
	\E_{n \in [P]^d} \prod_{i \in [t]} 
	\Lambda_{\chi,R,W} \big[ \psi_i(n) \big]
	= c_{\chi,2}^t + O( (\log N)^{-1+o(1)}).
\end{align}
By Lemma~\ref{thm:gpy:sievefactor}, we have $c_{\chi,2} > 0$
and therefore we may define a renormalized weight 
$\nu \coloneqq c_{\chi,2}^{-1} \Lambda_{\chi,R,W}$, 
which satisfies the desired pseudorandomness asymptotic 
by~\eqref{eq:gpy:prefinalpseudord},
and which majorizes a constant multiple of $\lambda_{b,W}$
by Lemma~\ref{thm:gpy:gpymajorizes}. 
\qed

\section{Quantitative pseudorandomness}
\label{sec:rd}

The goal of this section is to transfer the previous
pseudorandomness asymptotics over $\Z$ to the
setting of a large cyclic group,
and to show that pseudorandomness is preserved
under certain averaging operations.
We also state the generalized Von Neumann theorem
of Green and Tao~\cite[Appendix~C]{GT:lineqs},
in a quantified form.
The relevant notion of pseudorandomness
in our paper is the following.

\begin{definition}[Quantitative pseudorandomness]
\label{thm:rd:pseudorddef}
	Let $D,H \geq 1$ be parameters and let $M$ be a prime. 
	We say that $\nu : \Z_M \rightarrow \R^+$ is $D$-pseudorandom 
	of level $H$ when, for every affine system
	$\theta : \Z_M^d \rightarrow \Z_M^t$ of finite complexity
	such that $d,t,\| \dot{\theta} \| \leq D$,
	\begin{align*}
		\E_{n \in \Z_M^d} \nu\big[ \theta_1(n) \big] \dots \nu\big[ \theta_t(n) \big]
		= 1 + O_D \Big( \frac{1}{H} \Big).
	\end{align*}
\end{definition}

We now let $N$ denote an integer larger than some absolute constant,
and as in the previous section
we fix $\omega = c_0\log N$ and $W = \prod_{p \leq \omega} p$.
We also consider an embedding $[N] \hookrightarrow \Z_M$,
where $M$ is a prime larger than $N$.
We are then interested in finding a pseudorandom majorant over $\Z_M$
for the function $\lambda_{b,W}$ from 
Definition~\ref{thm:gpy:wprimesmeasuredef},
properly extended to a function on $\Z_M$.
Precisely, given a function $f : \Z \rightarrow \C$
with support in $[N]$, 
we define an $M$-periodic 
function $\wt{f}$ at $n \in \Z$ by $\wt{f}( n ) = f( n + \ell M )$,
where $\ell$ is the unique integer such that $n + \ell M \in [M]$,
and that function $\wt{f}$ may in turn be viewed as a function on $\Z_M$.

It is actually relatively simple to construct a pseudorandom majorant 
on $\Z_M$ from the one of
Proposition~\ref{thm:gpy:quantpseudordoverZ},
by cutting $\Z_M^d$ into small boxes as explained 
in~\cite[p.~527]{GT:apsinprimes}.
We rerun this argument here since we need 
to extract explicit error terms from it.

\begin{proposition}[Pseudorandom majorant over $\Z_M$]
\label{thm:rd:pseudordmajorant}
	Let $D \geq 1$.
	There exists a constant $C_D$ such that
	if $N \geq C_D$ and $M \geq N$ is a prime, 
	there exists a $D$-pseudorandom
	weight $\wt{\nu} : \Z_M \rightarrow \R^+$ 
	of level $(\log N)^{1-o(1)}$ such that 
	\begin{align*}
		0 \leq \wt{\lambda}_{b,W} \ll_D \wt{\nu}.
	\end{align*}
\end{proposition}

\begin{proof}
Consider an affine system $\theta : \Z_M^d \rightarrow \Z_M^t$
of finite complexity and such that
$d,t,\| \dot{\theta} \| \leq D$.
By Fact~\ref{thm:linalg:liftfincplx},
we may consider $\theta$ as the
reduction modulo $M$ of 
an affine system $\psi : \Z^d \rightarrow \Z^t$
with norms $\| \psi \|_M = \| \theta \| \leq 2D$
and $\| \dot{\psi} \| =  \| \dot{\theta} \| \leq D$. 
We let further implicit constants depend on $D$ 
in the course of this proof.

Let $\nu$ be the weight from Proposition~\ref{thm:gpy:quantpseudordoverZ},
and define $\wt{\nu} : \Z_M \rightarrow \R^+$ as above.
Choosing another scale $P = M^{1/2}$, 
and duplicating the variable of averaging, 
we obtain
\begin{align}
\label{eq:gpy:variableduplication}
	\E_{n \in [M]^d} \textstyle\prod\limits_{i \in [t]} \wt{\nu} \big[ \psi_i(n) \big]
	= \E_{m \in [M]^d} \E_{n \in [P]^d}
	\textstyle\prod\limits_{i \in [t]} 
	\wt{\nu}\big[ \psi_i(m + n) \big] + O( N^{-1/4} ).
\end{align}
We call an integer $m$ \textit{good} when
$\psi( m + [P]^d ) \subset [M]^t + M \ell$ for some $\ell \in \Z^t$,
and when that is not the case we say that $m$ is \textit{bad}.
When $m$ is good we have, with $\ell \in \Z^t$ as prescribed 
and by~\eqref{eq:gpy:pseudordoverZ}, 
\begin{align}
	\notag
	\E_{n \in [P]^d}
	\textstyle\prod\limits_{i \in [t]} \wt{\nu} \big[ \psi_i(m + n) \big]
	&= 
	\E_{n \in [P]^d}	
	\textstyle\prod\limits_{i \in [t]} \nu \big[ \dot{\psi}_i(n) + (\psi_i(m) - M\ell_i) \big] \\
	\label{eq:gpy:goodboxes}
	&= 1 + O_D( (\log N)^{-1+o(1)} ).
\end{align}
When $m$ is bad, we have
$\min_{i \in [t]} d( \psi_i(m) , M\Z ) \leq \| \dot{\psi} \| P$ 
with respect to the canonical distance 
$d(x,y) = | x - y |$ on $\R$.
Indeed, when that inequality does not hold, we have
\begin{align*}
	\psi(m+]0,P[^d) \cap 
	\{ y \in \R^t : \exists i \in [t] \ \text{such that}\ y_i \in M\Z \} 
	= \varnothing,
\end{align*}
and since $\psi(m+]0,P[^d)$ is connected
it must be contained in one of the
boxes $]0,M[^t + M \ell$, $\ell \in \Z^t$
(it is helpful to draw a picture at this point).
We have thus proven that when $m$ is bad,
there exists $i \in [t]$ and 
$\ell_i \in \Z$ such that
$\psi_i(m) \in \ell_i M + [-O(P),O(P)]$,
and such an $\ell_i$ is necessarily $\ll 1 + \| \psi \|_M \ll 1$.
It is easy to check that the number of such $m \in [M]^d$
is $\ll PM^{d-1} = M^{d-1/2}$.
Inserting the estimate~\eqref{eq:gpy:goodboxes}
on good-boxes averages in~\eqref{eq:gpy:variableduplication},
and neglecting the count of bad-boxes averages,
we obtain the desired asymptotic.
\end{proof}

The notion of pseudorandomness is quite robust
under averaging operations, as demonstrated by
the following proposition, which is needed later on
to majorize certain convolutions of $\lambda_{b,W}$.

\begin{proposition}
\label{thm:rd:pseudordofavg}
	Let $D,H \geq 1$ be parameters and $M$ be a prime.
	Suppose that
	$\nu : \Z_M \rightarrow \R^+$
	is $D$-pseudorandom of level $H$,
	$B$ is a symmetric subset of $\Z_M$
	and $\mu_B = (|B|/M)^{-1} 1_B$.
	Then $\nu' = \tfrac{1}{2}( \nu + \nu \ast \mu_B)$
	is also $D$-pseudorandom of level $H$.
\end{proposition}

\begin{proof}
Consider an affine system $\theta : \Z_M^d \rightarrow \Z_M^t$
of finite complexity such that $d,t,\|\dot{\theta}\| \leq D$.
Let $\nu^{(0)} = \nu$ and $\nu^{(1)} = \nu \ast \mu_B$,
so that $\nu^{(\eps)}(x) = \E_{y \in B} \nu( x + \eps y )$
for every $\eps \in \{0,1\}$ and $x \in \Z_M$.
Therefore
\begin{align*}
	S &\coloneqq \E_{n \in \Z_M^d} \tfrac{\nu^{(0)} + \nu^{(1)}}{2} \big[ \theta_1(n) \big] \cdots
	\tfrac{\nu^{(0)} + \nu^{(1)}}{2} \big[ \theta_t(n) \big] \\
	&= \E_{\eps \in \{0,1\}^t} \E_{n \in \Z_M^d} 
	\nu^{(\eps_1)}\big[ \theta_1(n) \big]
	\cdots 	\nu^{(\eps_t)}\big[ \theta_t(n) \big] \\
	&= \E_{\eps \in \{0,1\}^t} \E_{y \in B^t} \E_{n \in \Z_M^d}
	\nu\big[ \theta_1(n) + \eps_1 y_1 ] \cdots
	\nu\big[ \theta_t(n) + \eps_t y_t ].
\end{align*}
For every $\eps \in \{0,1\}^t$ and $y \in B^t$,
the system $(\theta_i + \eps_i y_i)_{1 \leq i \leq t}$
has same linear part as $(\theta_i)_{1 \leq i \leq t}$.
Since $\nu$ is $D$-pseudorandom of level $H$, we have
$S = 1 + O_D( H^{-1} )$ as desired.
\end{proof}

We now quote the generalized Von Neumann
theorem of Green and Tao~\cite[Appendix C]{GT:lineqs}.
It is simple to quantify the error term in that result
in terms of the level of pseudorandomness of the weight.

\begin{theorem}[Generalized Von Neumann theorem]
\label{thm:rd:genvonnm}
	Let $d,t,Q,H \geq 1$ and $s \geq 0$ be parameters,
	and let $i \in [t]$ be an indice.
	There exists a constant $D$ depending
	on $d,t,Q$ such that the following holds.
	Suppose that $M > D$ is a prime
	and $\theta : \Z_M^d \rightarrow \Z_M^t$
	is an affine system of finite complexity
	in exact $s$-normal form at $i$, and
	such that $\| \dot{\theta} \| \leq Q$.
	Suppose also that 
	$\nu : \Z_M \rightarrow \R^+$ is $D$-pseudorandom of level $H$,
	and $f_1,\dots,f_t : \Z_M \rightarrow \R$ 
	are functions such that $|f_j| \leq \nu$ for every $j \in [t]$.
	Then we have
	\begin{align*}
		\big| \E_{n \in \Z_M^d} 
		f_1 \big[ \theta_1(n) \big]
		\cdots
		f_t \big[ \theta_t(n) \big]
		\big|^{2^{s+1}}
		\leq \| f_i \|_{U^{s+1}(\Z_M)}^{2^{s+1}} + O_D(H^{-1}).
	\end{align*}
\end{theorem}

\begin{proof}
Up to relabeling the $f_j$ and $\theta_j$, 
we may assume that $i=1$.
Up to permutating the base vectors,
we may also assume that the set $J_1$ from
Definition~\ref{thm:linalg:normaldef}
is equal to $[s+1]$.
It then suffices to apply~\cite[Proposition~7.1'']{GT:lineqs},
whose proof invokes twice the pseudorandomness condition
of Definition~\ref{thm:rd:pseudorddef},
under the name ``linear forms condition''.
Note that the argument there 
requires a change of variable
$(x_1,\dots,x_{s+1},y) \mapsto (c_1^{-1} x_1,\dots,c_{s+1}^{-1} x_{s+1},y)$
with respect to the decomposition 
$\Z_M^d = \Z_M^{s+1} \times \Z_M^{d-(s+1)}$,
where $c_k = \dot{\theta}_1(e_k)$.
The condition $M > D \geq \| \dot{\theta} \|$ ensures that
this is possible, however the new forms involved 
may have large size, potentially
not bounded in terms of $\| \dot{\theta} \|$.
Fortunately, it can be verified that making the change of variables
$x_i \mapsto c_i c_{s+1} x_i$, $1 \leq i \leq s+1$
before each application of the linear forms condition 
in the proof of~\cite[Proposition~7.1'']{GT:lineqs}
converts the systems of forms under consideration
back into sytems of bounded size.
(Here we elaborated slightly on the footnote at 
the bottom of~\cite[p.~1822]{GT:lineqs}).
\end{proof}

\section{Translation-invariant equations in the primes}
\label{sec:tsf}

In this Section, we prove Theorem~\ref{thm:intro:primescplx1}.
Our two main tools are the 
transference principle of
Helfgott and de Roton~\cite{HR:rothprimes},
including Naslund's~\cite{Nasl:rothprimes} refinement thereof, 
and the relative generalized Von Neumann
theorem of Green and Tao,
in the quantitative form obtained
in the previous section.
These two tools together transfer the problem of finding
a complexity-one pattern in the primes, to that of finding
one in the integers, and to finish the proof we simply apply
our extension of Shao's result
derived in Appendix~\ref{sec:apx:intg}.

We now formally begin the proof of Theorem~\ref{thm:intro:primescplx1}.
We start with a standard preliminary reduction,
the $W$-trick, which allows us to consider subsets
of an arithmetic progression of modulus $W$ 
in the primes instead.

\begin{theorem}[Theorem~\ref{thm:intro:primescplx1} in $W$-tricked primes]
\label{thm:tsf:wtrickedcplx1}
	Let $V \in \M_{r \times t}(\Z)$ be a translation-invariant matrix
	of rank $r$ and complexity one.
	There exists a constant $C$ depending at most on $r,t,V$
	such that the following holds.
	Let $W = \prod_{p \leq \omega} p$, 
	where $\omega = c_0 \log N$ with 
	$c_0 \in [\tfrac{1}{4},\tfrac{1}{2}]$,
	and let $b \in \Z$ such that $(b,W)=1$.
	Suppose that $A$ is a subset of $[N]$ such that 
	$b + W \cdot A \subset \mathcal{P}$ and
	\begin{align*}
		|A| &= \alpha (W/\phi(W))(\log N)^{-1} N, \\
		\alpha &\geq C(\log\log N)^{-1/25t}.
	\end{align*}
	Then there exists $\by \in A^t$ with distinct coordinates
	such that $V \by = 0$.
\end{theorem}

\textit{Proof that Theorem~\ref{thm:tsf:wtrickedcplx1} 
implies Theorem~\ref{thm:intro:primescplx1}}.

Consider a subset $A$ of $\mathcal{P}_N$ of density $\alpha$;
we may certainly assume that $\alpha \geq CN^{-1/4}$,
and in particular that $N$ is large enough.
Let $W = \prod_{p \leq \omega} p$, 
where $\omega = \frac{1}{4} \log N$,
and let $N' = \lfloor N/W \rfloor = N^{3/4+o(1)}$ 
(by the prime number theorem) be another scale.
By~\cite[Lemma~2.1]{HR:rothprimes}, there exists
$(b,W) = 1$ such that
$A' = \{ n \in [N'] : b + W n \in A \}$
has size $\gg \alpha (W/\phi(W)) (\log N')^{-1} N'$.
Note that $\omega \sim \frac{1}{3} \log N'$
as $N \rightarrow \infty$,
and since $b + W \cdot A' \subset A$,
every solution $\by \in (A')^t$ to $V\by = 0$
with distinct coordinates induces one in $A^t$, 
by translation-invariance and homogeneity.
Applying then Theorem~\ref{thm:tsf:wtrickedcplx1} 
to $A' \subset [N']$ concludes the proof.
\qed

\medskip

From now on, we work under the hypotheses
of Theorem~\ref{thm:tsf:wtrickedcplx1}.
First, we fix a translation-invariant
matrix $V \in \M_{r \times t}(\Z)$
of complexity one,
and without loss of generality
we may assume that $t \geq 3$ and $V$ has no zero columns.
Via Propositions~\ref{thm:linalg:kernelparam}
and~\ref{thm:linalg:exact1normalform},
we can choose a linear parametrization
$\psi : \Z^d \twoheadrightarrow \Z^t \cap \Ker_\Q(V)$
in exact $1$-normal form over $\Z$ at every $i \in [t]$.
We assume from now on that $N$ is large enough with
respect to $d,t,\psi,V$,
and we let further implicit and explicit constants
depend on those parameters.
We will need to consider functions 
with support in $[-2N,2N]_\Z$,
and to analyze those we embed $[-2N,2N]_\Z$ 
in a large cyclic group $\Z_M$,
where $M$ is a prime between 
$4(\|V\| + 1) \cdot N$ and $8(\|V\| + 1) \cdot N$
chosen via Bertrand's postulate.
By Fact~\ref{thm:linalg:reducenormal},
the linear map $\psi$ reduces modulo $M$ to a linear map
$\theta : \Z_M^d \twoheadrightarrow \Ker_{\Z_M}(V)$
in exact $1$-normal form over $\Z_M$ at every $i \in [t]$, 
and such that $\| \theta \| = \| \psi \|$;
we work exclusively with that map from now on.

Next, we consider an integer $N \geq 1$ and 
a constant $c_0 \in [\tfrac{1}{4},\tfrac{1}{2}]$, 
and we fix 
\begin{align*}
	W = \prod_{p \leq \omega} p,
	\quad\quad
	\omega = c_0 \log N,
	\quad\quad
	b \in \Z : (b,W) = 1.
\end{align*}
We then consider a subset $A \subset [N]$ such that 
$|A| = \alpha \frac{W}{\phi(W)} (\log N)^{-1} \cdot N$
and $b + W \cdot A \subset \mathcal{P}$.
Accordingly, we define the normalized
indicator function of $A$ by
\begin{align*}
	\lambda_A = L \frac{\phi(W)}{W} (\log N) \cdot 1_A,
\end{align*}
where $L = M/N \asymp 1$.
With this normalization, we have $\E_{[M]} \lambda_A = \alpha$ and
$0 \leq \lambda_A \ll \lambda_{b,W}$,
recalling Definition~\ref{thm:gpy:wprimesmeasuredef}.

Given a function $f : \Z \rightarrow \C$
with support in $[-2N,2N]$,
we define an $M$-periodic function $\breve{f}(n) = 0$
at $n \in \Z$ by $\breve{f}(n) = f( n + \ell M)$,
where $\ell$ is the unique integer
such that $n + \ell M \in [-M/2,M/2]_\Z$,
and $\breve{f}$ may then be considered as a function on $\Z_M$.
When $f$ has support in $[N]$, as is the case
for $\lambda_{b,W}$, this coincides
with the definition of $\wt{f}$ from Section~\ref{sec:rd}.
To alleviate the notation, we now identify
functions $f : \Z \rightarrow \C$
with support in $[-2N,2N]$ with their 
periodic counterpart $\breve{f}$.
Most of the analysis we do next takes place 
on $\Z_M$, and Fourier transforms,
convolutions, $L^p$ and $U^k$ norms are 
normalized accordingly.
With these notations in place, we now work 
with the following pattern-counting operator.

\begin{definition}
\label{thm:tsf:ToperatoroverZM}
	We define the operator $T$ on functions $f_1,\dots,f_t : \Z_M \rightarrow \R$ by
	\begin{align*}
		T( f_1, \dots, f_t ) = 
		\E_{n \in \Z_M^d} f_1 \big[ \theta_1(n) \big] \dots f_t \big[ \theta_t(n) \big].
	\end{align*}
\end{definition}

If need be, we can always return to averages over $\Z$ via the following observation.

\begin{lemma}
\label{thm:tsf:ToperatoroverZ}
	For functions $f_1,\dots,f_t : \Z_M \rightarrow \R$
	with support in $[-2N,2N]$, we have
	\begin{align*}
		T( f_1, \dots, f_t )
		= M^{-(t-r)} \sum_{\substack{y \in [-2N,2N]_\Z^t \,: \\ Vy = 0}}
		f_1 (y_1) \dots f_t(y_t).
	\end{align*}
\end{lemma}

\begin{proof}
Since $\theta$ is a surjection onto
$\Ker_{\Z_M}(V)$, and the fibers 
$\#\{ x \in \Z_M^d : \theta(x) = y \}$
have uniform size when $y$ ranges over $\Ker_{\Z_M}(V)$, we have
\begin{align*}
	T( f_1, \dots, f_t )
	&= \E_{y \in \Z_M^t : V y = 0} f_1 (y_1) \dots f_t(y_t) \\
	&= M^{-(t-r)} \textstyle\sum_{y \in \Z_M^t : V y = 0} f_1 (y_1) \dots f_t(y_t).
\end{align*}
Since the $f_i$ have support in $[-2N,2N]$,
we may restrict the summation to $y \in [-2N,2N]_\Z^t$,
and since $M > 2\| V \| N$, the identity $Vy = 0$
holds in $\Z$ for such $y$.
\end{proof}

We now introduce two parameters $\delta \in (0,1]$
and $\eps \in (0,c\,]$.
We also fix an auxiliary Bohr set of $\Z_M$
(see Definition~\ref{thm:intg:bohrsetdef})
defined by
\begin{align*}
	\Gamma &= \{ r \in \Z_M : |\wh{\lambda}_A(r)| \geq \delta  \} \cup \{1\}, \\
	B &= B( \Gamma , \eps ).
\end{align*}
The presence of $1$ in the frequency set
guarantees that the Bohr set is contained
in an interval $[-\eps M,\eps M]$.
As is common in the transference literature for
three-term arithmetic progressions~\cite{G:rothprimes,GT:restriction,HR:rothprimes,Nasl:rothprimes},
we work with a smooth approximation of $\lambda_A$, 
namely the convolution over $\Z$ given by
\begin{align*}
	\lambda'_A = \lambda_A \ast \lambda_B,
\end{align*}
where $\lambda_B = |B|^{-1} 1_B$. 
Provided that $\eps$ is small enough,
we see that the support of $\lambda'_A$
is contained in $[-2N,2N]$.
Since $M > 2N$, we may also consider 
$\lambda'_A : \Z_M \rightarrow \R$
as the normalized convolution over $\Z_M$ given by
\begin{align}
\label{eq:tsf:omegaAdef}
	\lambda'_A = \lambda_A \ast \mu_B,
\end{align}
where $\mu_B = (|B|/M)^{-1} 1_B$.
To show that $\lambda'_A$ is close to $\lambda_A$ in a Fourier $\ell^4$ sense,
we need to call on the restriction estimates 
of Green and Tao~\cite{GT:restriction},
themselves based on an envelopping sieve of 
Ramaré and Ruzsa~\cite{RR:additive};
these estimates were in turn specialized
to the case of a large modulus $\omega$
by Helfgott and de Roton~\cite{HR:rothprimes},
and an alternative approach to those 
can be found in a blog post of Tao~\cite{tao:restrictionext}.

\begin{proposition}
\label{thm:tsf:U2bound}
	We have $\| \lambda_A - \lambda'_A \|_{U^2} \ll \eps^{1/4} + \delta^{1/4}$.
\end{proposition}

\begin{proof}
	By~\cite[Lemma~2.2]{HR:rothprimes},
	we have $\sum_r |\wh{\lambda}_A(r)|^q \ll_q 1$ for any $q > 2$.
	Therefore,
	\begin{align*}
		\| \lambda_A - \lambda'_A \|_{U^2}^4
		&= \sum_r |\wh{\lambda}_A(r)|^4 |1 - \wh{\mu}_B(r)|^4 \\
		&\ll \eps \sum_{r :\, |\wh{\lambda}_A(r)| \geq \delta} |\wh{\lambda}_A(r)|^4
		+ \delta \sum_{r :\, |\wh{\lambda}_A(r)| \leq \delta } |\wh{\lambda}_A(r)|^3 \\
		&\ll \eps + \delta,
	\end{align*}
	where we used the fact that 
	$|1-\wh{\mu}_B(r)| = | \E_{x \in B} ( 1 - e_N(rx) ) | \leq 2\pi \eps$ 
	for all $r \in \Gamma$.
\end{proof}

The structure of our argument is now as follows:
we compare the counts $T(\lambda_A,\dots,\lambda_A)$
and $T(\lambda'_A,\dots,\lambda'_A)$, which we expect to be
close by Proposition~\ref{thm:tsf:U2bound} and the
heuristic that ``the $U^2$ norm controls complexity one averages".

\begin{remark}[Multilinear expansion]
\label{thm:tsf:multilinexpansion}
	By multilinearity,
	\begin{align}
	\label{eq:tsf:multilinexpansion}
		T( \lambda_A, \dots , \lambda_A )
		= T ( \lambda'_A, \dots, \lambda'_A )
		+ \sum T( \ast, \dots , \lambda_A - \lambda'_A , \dots , \ast).
	\end{align}
	where the sum is over $2^t - 1$ terms and the
	stars stand for functions equal to $\lambda'_A$
	or $\lambda_A - \lambda'_A$.
\end{remark}

To estimate the main term in~\eqref{eq:tsf:multilinexpansion},
that is, $T(\lambda'_A,\dots,\lambda'_A)$,
we invoke a key transference estimate of Helfgott and de Roton~\cite{HR:rothprimes}, 
which essentially allows us to consider $\lambda'_A$ as
a subset of the integers of density $\alpha^{2}$.
It is further possible, by a result of Naslund\footnote{
Here we implicitely refer to the first version of Naslund's preprint,
because the argument there is simpler, and we do not seek very sharp
bounds on the exponent.}~\cite{Nasl:rothprimes}, 
to obtain an exponent $1+o(1)$ instead of $2$, and we 
choose to work with that more efficient version,
even though it is possible to
derive Theorem~\ref{thm:intro:primescplx1} with a smaller exponent without it.
This is because we wish to exhibit that our argument preserves the exponent
in Szemerédi-type theorems in the integers, 
in the sense of Proposition~\ref{thm:tsf:conversion} below.

\begin{proposition}
\label{thm:tsf:omegaAdense}
	Suppose that $\delta^{-4} \log \eps^{-1} \leq c\log N$.
	Then for any $\kappa > 0$, 
	the level set ${A' = \{ \lambda'_A \geq \alpha/2 \}}$
	has density $\gg_\kappa \alpha^{1+\kappa}$ in $\Z_M$.
\end{proposition}

\begin{proof}
Recalling~\eqref{eq:tsf:omegaAdef},
we see that $\E \lambda'_A = \E \lambda_A = \alpha$. 
By Selberg's sieve or the restriction estimate used 
in the proof of Proposition~\ref{thm:tsf:U2bound}, we have 
\begin{align*}
 	\#\{ r : |\wh{\lambda}_A(r)| \geq \delta \} 
 	\leq \delta^{-4} \| \wh{\lambda}_A \|_4^4 \ll \delta^{-4},
\end{align*}
and therefore $|B| \geq \eps^{|\Gamma|} N \geq N^{1/2}$
under our assumptions on $\eps$ and $\delta$.
By~\cite[Proposition~2]{Nasl:rothprimes},
we deduce that $\| \lambda'_A \|_{p} \ll_p 1$ 
for any even $p \geq 4$,
and the proposition then follows from the $p$-th
moment version of the classical Paley-Zygmund inequality.
\end{proof}

Applying our statistical, complexity-one extension of 
Shao's result in the integers, we can now obtain
a lower bound on the average of $\lambda'_A$ 
over $\psi$-configurations.

\begin{proposition}[Main term]
\label{thm:tsf:maintermlowerbound}
	Suppose that $\delta^{-4} \log \eps^{-1} \leq c\log N$.
	We have
	\begin{align*}
		T( \lambda'_A,\dots,\lambda'_A ) 
		\geq \exp\big[ \! - C_\kappa \alpha^{ -24t - \kappa } \big]
	\end{align*}
	for every $\kappa > 0$.
\end{proposition}

\begin{proof}
Consider the level set $A' = \{ \lambda'_A \geq \alpha/2 \}$
contained in the support of $\lambda'_A$, 
and therefore in $[-2N,2N]$.
Since $\lambda'_A \geq (\alpha/2) \cdot 1_{A'}$, we have
\begin{align*}
	T( \lambda'_A, \dots, \lambda'_A ) \geq (\alpha/2)^t T( 1_{A'}, \dots , 1_{A'} ).
\end{align*}
By Proposition~\ref{thm:tsf:omegaAdense}, we know that
$A'$ has density $\gg_\kappa \alpha^{1+\kappa}$
in $[-2N,2N]$ for any $\kappa > 0$.
Invoking Lemma~\ref{thm:tsf:ToperatoroverZ},
and applying Proposition~\ref{thm:intg:intgcplx1} 
to $A' \subset [-2N,2N]$, we obtain
\begin{align*}
	T( 1_{A'} , \dots , 1_{A'} )
	= M^{-(t-r)} \#\{ y \in (A')^t : Vy = 0 \}
	\geq \exp\big[ - C_\kappa \alpha^{-(1+\kappa)24t} \,\big].
\end{align*}
\end{proof}

On the other hand, the averages from~\eqref{eq:tsf:multilinexpansion} 
involving a difference $\lambda_A - \lambda'_A$ are bounded via
the generalized Von Neumann theorem of Section~\ref{sec:rd}.

\begin{proposition}[Error terms]
\label{thm:tsf:errortermupperbound}
	Suppose that $f_1,\dots,f_t$
	are functions all equal to $\lambda'_A$ or $\lambda_A - \lambda'_A$,
	with at least one of them equal to $\lambda_A - \lambda'_A$.
	Then
	\begin{align*}
		| T( f_1, \dots, f_t ) | \ll \eps^{1/4} + \delta^{1/4} + (\log N)^{-\frac{1}{4}+o(1)}.
	\end{align*}
\end{proposition}

\begin{proof}
We consider $i \in [t]$ such that $f_i = \lambda_A - \lambda'_A$.
Let $Q = \| \dot{\theta} \|$ and 
let $D = D_{d,t,Q}$ be the constant from
Proposition~\ref{thm:rd:genvonnm}.
By Proposition~\ref{thm:rd:pseudordmajorant},
and since we assumed $N$ to be large enough with respect
to $d,t,\theta$, there exists a $D$-pseudorandom weight 
$\nu : \Z_M \rightarrow \R^+$ of level $(\log N)^{1-o(1)}$
such that 
\begin{align*}
	0 \leq \lambda_A \ll \lambda_{b,W} \ll \nu.
\end{align*}
Let $\nu' = \tfrac{1}{2}(\nu + \nu \ast \mu_B )$,
so that $|\lambda'_A| \ll \nu'$ and 
$|\lambda_A - \lambda'_A| \ll \nu'$.
By Proposition~\ref{thm:rd:pseudordofavg},
$\nu'$ is also $D$-pseudorandom of level $(\log N)^{1-o(1)}$.

Recall now that $\psi$ is in exact $1$-normal form at $i$.
Applying Proposition~\ref{thm:rd:genvonnm} with $s = 1$
to the functions $f_1,\dots,f_t$ 
(divided by a certain large constant),
and inserting the estimates of Proposition~\ref{thm:tsf:U2bound},
we obtain the desired bound.
\end{proof}

At this point we need only collect together
the bounds on the main term and the error terms
in~\eqref{eq:tsf:multilinexpansion}
to finish the proof of Theorem~\ref{thm:intro:primescplx1},
which we have previously reduced to
proving Theorem~\ref{thm:tsf:wtrickedcplx1}.
\medskip

\textit{Proof of Theorem~\ref{thm:tsf:wtrickedcplx1}.}
Starting from the multilinear 
expansion~\eqref{eq:tsf:multilinexpansion},
and inserting the bounds from
Propositions~\ref{thm:tsf:maintermlowerbound}
and~\ref{thm:tsf:errortermupperbound},
we obtain
\begin{align*}
	T( \lambda_A , \dots , \lambda_A ) 
	\geq \exp[ -C_\kappa  \alpha^{-24t - \kappa} ]
	- O\Big( \eps^{1/4} + \delta^{1/4} + (\log N)^{-\tfrac{1}{4} + o(1)} \Big),
\end{align*}
whenever, say, $\eps^{-1}, \delta^{-1} \leq c(\log N)^{1/8}$.
Choose now $\eps = \delta = \exp[ -C'_\kappa  \alpha^{-24t - \kappa} ]$
(for a large $C'_\kappa$),
and assume that $\alpha \geq C_\kappa(\log\log N)^{-1/(24t+\kappa)}$.
This ensures that the conditions on $\eps$ and $\delta$ are satisfied,
and that we have a lower bound
\begin{align*}
	T( \lambda_A , \dots , \lambda_A ) 
	\geq \exp[ -C'_\kappa  \alpha^{-24t - \kappa} ].
\end{align*}
By Lemma~\ref{thm:tsf:ToperatoroverZ} 
and since $\lambda_A \leq (\log N) 1_A$, we then have
\begin{align*}
	\#\{ y \in A^t : Vy = 0 \} 
	\geq \exp\big[- C_\kappa\alpha^{-24t - \kappa} \big] \cdot N^{t-r} (\log N)^{-t}.
\end{align*}
On the other hand, by Lemma~\ref{thm:linalg:trivsols},
the number of $y \in [N]^t$ with two identical coordinates 
and such that $Vy=0$ is $\ll N^{t-r-1}$.
Choosing now $\kappa = t$ for aesthetic reasons,
and given the range of density under consideration,
we are therefore ensured to find at least one non-trivial solution.
\qed

As claimed before, our argument 
allows for a slightly more general statement than 
Theorem~\ref{thm:intro:primescplx1}.
Indeed, the following can be obtained by 
a suitable Varnavides argument and by inserting
the resulting analog of 
Proposition~\ref{thm:intg:intgcplx1} in our proof.

\begin{theorem}
\label{thm:tsf:conversion}
	Suppose that $V \in \M_{r \times t}(\Z)$ is a translation-invariant matrix
	of rank $r$ and complexity one, and let $\gamma > 0$ be a parameter.
	Assume that $V \by = 0$ 
	has a distinct-coordinates solution $\by \in A^t$
	for every subset $A$ of $[N]$ of density at least
	\begin{align*}
		C(\log N)^{-\gamma}.
	\end{align*}
	Then such a solution also exists for
	every subset $A$ of $\mathcal{P}_N$ of density at least
	\begin{align*}
		C_\eps (\log\log N)^{-\gamma+\eps},
	\end{align*}
	for any $\eps > 0$.
\end{theorem}

This being said, we have not tried to optimize 
the exponent $1/24t$ in Corollary~\ref{thm:intg:intgcplx1orig},
nor the exponent in Theorem~\ref{thm:intro:primescplx1}
that follows from it.
This is because the former exponent is likely not
optimal, and far from comparable in quality with
Sanders'~\cite{S:roth} bounds for Roth's theorem, 
because of the repeated applications
of Cauchy-Schwarz in Appendix~\ref{sec:apx:intg}.

\appendix

\section{Translation-invariant equations in the integers}
\label{sec:apx:intg}

The purpose of this section is to derive an
extension of a result of Shao~\cite{Shao:configs}
to arbitrary systems of complexity one,
and with a count of the multiplicity of pattern occurences.
The structure of our proof is similar 
to Shao's, and it relies in particular in the key 
local inverse $U^2$ theorem proved there
(Proposition~\ref{thm:intg:U2inv} below).
However, certain added technicalities arise
when handling arbitrary systems:
the most significant of those is addressed 
by Proposition~\ref{thm:intg:untwistingU2} below.

\begin{proposition}
\label{thm:intg:intgcplx1}
	Let $V \in \M_{r \times t}(\Z)$ 
	be a translation-invariant matrix
	of rank $r$ and complexity one.
	Suppose that $A$ is a subset of $[-N,N]_\Z$
	of density $\alpha$.
	Then
	\begin{align*}
		\# \{ \by \in A^t : V \by = 0 \} 
		 \geq \exp\big[ \! - C \alpha^{-24t} \big] \cdot N^{t-r},
	\end{align*}
	for a constant $C > 0$ depending at most on $r,t,V$.
\end{proposition}

Although we only need the result above
for the transference argument
of Section~\ref{sec:tsf}, 
we record the following consequence,
since it may be of independent interest.

\begin{corollary}
\label{thm:intg:intgcplx1orig}
	Let $V \in \M_{r \times t}(\Z)$ 
	be a translation-invariant matrix
	of rank $r$ and complexity one.
	There exists a constant $C > 0$ depending
	at most on $r,t,V$ such that,
	if $A$ is a subset of $[N]$ of density at least
	$C(\log N)^{-1/24t}$,
	there exists a solution $\by \in A^t$
	to $V \by = 0$ with distinct coordinates.
\end{corollary}

\begin{proof}
By Lemma~\ref{thm:linalg:trivsols},
the number of $\by \in [N]^t$ with two equal coordinates
such that $V \by = 0$ is at most $O(N^{t-r-1})$.
The result then follows from Proposition~\ref{thm:intg:intgcplx1},
since we assumed that $\alpha \geq C(\log N)^{-1/24t}$.
\end{proof}

We now fix a translation-invariant matrix 
$V \in \M_{r \times t}(\Z)$ of rank $r$,
and for the purpose of proving Proposition~\ref{thm:intg:intgcplx1},
we may assume without loss of generality
that $t \geq 3$ and $V$ has no zero columns.
By Propositions~\ref{thm:linalg:kernelparam}
and~\ref{thm:linalg:exact1normalform}, 
we may choose a linear parametrization
$\varphi : \Z^{q+1} \twoheadrightarrow \Z^t \cap \Ker_{\Q}(V)$
of the form $\varphi(x_0,x) = x_0 \mathbf{1} + \psi(x)$,
where $\psi : \Z^q \rightarrow \Z^t$ is in 
exact $1$-normal form at every $i \in [t]$.
We have traded the letter $d$ for $q$ here
because the former is too precious 
as the dimension of a Bohr set.
Writing  $\psi_i(x) = a_{i1} x_1 + \dots + a_{iq} x_q$,
we define the sets of non-zero coefficients 
${ \Xi_i = \{ a_{ij} \neq 0,\, j \in [q] \} }$
and $\Xi = \cup_{i \in [t]} \Xi_i$,
so that we have $|a| \leq \| \varphi \|$ for
every $a \in \Xi$.

We also consider a fixed integer $N$ from the statement 
of Proposition~\ref{thm:intg:intgcplx1}, 
which should be thought of as quite large.
As usual, we choose to carry out our Fourier analysis 
over a cyclic group $\Z_M$ on a slightly larger scale;
to be precise, via Bertrand's postulate we pick a prime $M$ such that 
$\|\varphi\| \cdot 2N < M \leq \| \varphi \| \cdot 4N$.
Finally, throughout this section the letters $c$ and $C$ denote 
positive constants which are chosen, respectively, 
small or large enough with respect to $q$, $t$ and $\varphi$.
While we do not attempt to track the dependency of 
our parameters on $\| \varphi \|$, we sometimes use 
this quantity to illustrate our argument.

We now recall the basics of Bohr sets and regularity calculus,
which can be found in many places~\cite{me:bourgainroth,AG:acbook,GT:U3inv}.
We speed up this process as this material is utterly standard
and our notation is consistent with the litterature.

\begin{definition}
\label{thm:intg:bohrsetdef}
	A Bohr set of frequency set $\Gamma \subset \Z_M$
	and radius $\delta > 0$ is 
	\begin{align*}
		B(\Gamma,\delta) = \{ x \in \Z_M : \| \tfrac{xr}{M} \| \leq \delta \quad \forall r \in \Gamma \},
	\end{align*}
	and its dimension $d$ is defined by $d=|\Gamma|$.
	We often let the parameters $\Gamma,\delta,d$ be implicitely defined
	whenever we introduce a Bohr set $B$.
	The $\rho$-dilate $B_{| \rho}$ of a Bohr set $B$ is defined
	by $B(\Gamma,\delta)_{| \rho} = B(\Gamma,\rho \delta)$,
	and given two Bohr sets $B,B'$ we write $B' \leq_{\rho} B$
	when $B' \subset B_{| \rho}$.
	Finally, we say that $B$ is \textit{regular} when, for every $0 < \rho \leq 2^{-6} / d$,
	\begin{align*}
		( 1 - 2^6 \rho d ) |B| \leq
		|B_{| 1 \pm \rho}| \leq ( 1 + 2^6 \rho d ) |B|.
	\end{align*}
\end{definition}

We also recall standard size estimates on Bohr sets,
as well as Bourgain's regularization lemma.
In our later argument, all Bohr sets will be picked regular.

\begin{fact}
	Suppose that $B$ is a Bohr set of dimension $d$ and radius $\delta$,
	and $\rho \in (0,1]$. Then
	\begin{align*}
		|B| \geq \delta^d M 
		\quad\text{and}\quad
		|B_{| \rho}| \geq (\rho/2)^{2d} |B| .
	\end{align*}
	Given any Bohr set $B$, there exists $c \in [\tfrac{1}{2},1]$ 
	such that $B_{|c}$ is regular.
\end{fact}

In practice, regularity is used
in the following form, 
close in spirit to~\cite[Lemma~4.2]{GT:U3inv}.
When we argue ``by regularity'' in a proof, 
we implicitely invoke these estimates.

\begin{fact}[Regularity calculus]
	Let $f : \Z_M \rightarrow [-1,1]$
	and suppose that $B$ is a regular $d$-dimensional Bohr set,
	$X' \subset B_{| \rho}$ is another set and
	$x' \in B_{| \rho}$, where $\rho \in ( 0 , c/d \,]$.
	Then
	\begin{align*}
		\E_{x \in x' + B} f(x) &= \E_{x \in B} f(x) + O( \rho d ), \\
		\E_{x \in B} f(x) &= \E_{x \in B, x' \in X'} f(x + x') + O( \rho d ), \\
		\E_{x \in B} 1( x \in B_{| 1-\rho} ) f(x) &= \E_{x \in B} f(x) + O( \rho d ) .
	\end{align*}	 
\end{fact}

Before proceeding further, we recall certain facts about
Gowers box norms~\cite[Appendix B]{GT:lineqs},
which are present in disguise in Shao's argument~\cite{Shao:configs}.
For our argument, we only require
the positivity of such norms,
and two Cauchy-Schwarz-based inequalities.
Strictly speaking, we could do without those norms,
however they are useful 
to write averages over cubes in a more compact 
(if less intuitive) form,
and to expedite repeated applications of Cauchy-Schwarz.
In the following definitions, we let $X_1,X_2$
denote arbitrary subsets of $\Z_M$.

\begin{definition}[Box scalar product and norm]
	The box scalar product of a family of functions
	$(h_{\omega} : X_1 \times X_2 \rightarrow \R)_{\omega \in \{0,1\}^2}$ is
	\begin{align*}
		\langle (h_\omega) \rangle_{\Box(X_1 \times X_2)}
		= \E_{x^{(0)},x^{(1)} \in X_1 \times X_2} 
		\textstyle\prod\limits_{ \omega \in \{0,1\}^2 }
		h_\omega( x_1^{(\omega_1)} , x_2^{(\omega_2)} ).	
	\end{align*}
	The box norm of a function
	$h : X_1 \times X_2 \rightarrow \R$ is defined by
	$\| h \|_{\Box(X_1 \times X_2)}^4 = 
	\langle (h) \rangle_{\Box(X_1 \times X_2)}$.
\end{definition}

The first inequality we require is a box Van der Corput
inequality implicit in~\cite[p.~161]{GW:complexity}, while the
second is the Gowers-Cauchy-Schwarz 
inequality~\cite[Lemma~B.2]{GT:lineqs}.

\begin{fact}
	For $h : X_1 \times X_2 \rightarrow \R$
	and $(b_k : X_k \rightarrow [-1,1])_{k \in \{1,2\}}$,
	we have
	\begin{align}
	\label{eq:intg:boxVDC}
		\big| \E_{x_1 \in X_1,x_2 \in X_2} h(x_1,x_2) b_1(x_1) b_2(x_2) \big|
		\leq \| h \|_{\Box(X_1 \times X_2)}.
	\end{align}
	For $(h_{\omega} : X_1 \times X_2 \rightarrow \R)_{\omega \in \{0,1\}^2}$, we have
	\begin{align}
	\label{eq:intg:GCS}
		\big| \langle ( h_\omega ) \rangle_{\Box(X_1 \times X_2)} \big|
		\leq \prod_{\omega \in \{0,1\}^2}
		\| h_\omega \|_{\Box(X_1 \times X_2)}.
	\end{align}
\end{fact}

In our situation, we need a slight variant of the local
$U^2$ norm defined in~\cite{Shao:configs}.

\begin{definition}[Twisted $U^2$ norm]
\label{thm:intg:twistedU2def}
	Let $a,b \in \Z$ and $g : \Z_M \rightarrow \R$.
	The $(a,b)$-twisted $U^2$ norm of $g$ with respect to $X_1,X_2$ is
	\begin{align*}
		\| g \|_{\boxtimes_{a,b}(X_1 \times X_2)}^4 =
		\E_{x^{(0)},x^{(1)} \in X_1 \times X_2} 
		\textstyle\prod\limits_{ \omega \in \{0,1\}^2 }
		g( a x_1^{(\omega_1)} + b x_2^{(\omega_2)} ).		
	\end{align*}
	When $a=b=1$ we simply write $\| g \|_{\boxtimes(X_1 \times X_2)}$.
\end{definition}

With these notations, the local Gowers norm of a function $f$
with respect to sets $X_0,X_1,X_2$ as defined by 
Shao~\cite[Definition~3.1]{Shao:configs} is
\begin{align*}
	\| f \|_{ U^2(X_0,X_1,X_2) }^4 = 
	\E_{x_0 \in X_0} \| f( x_0 + \,\cdot\,) \|_{\boxtimes(X_1 \times X_2)}^4.
\end{align*}
From now on we keep the suggestive 
``local Gowers norm'' terminology,
but we use the expression in the right-hand side 
for computational purposes.

We are now ready to start with the proof
of Proposition~\ref{thm:intg:intgcplx1}.
We introduce, for a system of Bohr sets 
$\bB = (B_0,\dots,B_q)$,
the multilinear operator on functions
\begin{align*}
	T_{\bB} (f_1,\dots,f_t) = 
	\E_{x_0 \in B_0,\dots,x_q \in B_q} 
	f_1 \big[ \varphi_1 (x) \big] \dots f_t \big[ \varphi_t (x) \big].
\end{align*}
The next proposition then constitutes the first step of
our density increment strategy, in which we deduce
that a set $A$ either possesses many $\varphi$-configurations,
or it induces a large $T_{\bB}$-average involving
the balanced function of $A$.
Here and in the following, we occasionally make superfluous assumptions 
on the Bohr sets involved, in order to facilitate the combination 
of intermediate propositions.

\begin{proposition}[Multilinear expansion]
\label{thm:intg:multilinexp}
	Suppose that $A$ is a subset of density $\alpha$
	of a regular $d$-dimensional
	Bohr set $B = B_0$,
	and write $f_A = 1_A - \alpha 1_B$.
	Suppose also that $B_1,\dots,B_q$ are regular Bohr sets
	with $B_i \leq_{\rho} B_{i-1}$ for all $i \in [q]$, 
	where $\rho \leq c/d$.
	Then either
	\begin{enumerate}
		\item (Many patterns) $T_{\bB}(1_A,\dots,1_A) \geq \alpha^t / 4$,
		\item (Large $T$-average) 
				or there exist functions 
				$f_1,\dots,f_t : \Z_M \rightarrow [-1,1]$
				and $i \in [t]$ such that $f_i = f_A$ and
				$|T_{\bB}(f_1,\dots,f_t)| \gg \alpha^t$.
	\end{enumerate}
\end{proposition}

\begin{proof}
First observe that, expanding $1_A = \alpha 1_B + f_A$ by multilinearity,
\begin{align}
	\label{eq:intg:multilinexp}
	T_{\bB}(1_A,\dots,1_A) 
	= T_{\bB}(\alpha 1_B,\dots,\alpha 1_B)
	+ \sum T_{\bB}(\ast,\dots,f_A,\dots,\ast)
\end{align}
where the sum is over $2^t - 1$ terms
and the stars stand for functions equal to $\alpha 1_B$ or $f_A$. 
By definition,
\begin{align*}
	T_{\bB}(\alpha 1_B,\dots,\alpha 1_B)
	= \alpha^t \E_{x_0 \in B} \E_{x \in B_1 \times \dotsb \times B_q}
	1_B \big[ x_0 + \psi_1(x) \big] \dots 1_B \big[ x_0 + \psi_t (x) \big].
\end{align*}
Restricting $x_0$ to lie in $B_{| 1-\rho}$ with $\rho \leq c/\|\varphi\|d$, 
we are ensured that $x_0 + \psi_j(x) \in B$ for every $j \in [t]$
and $x \in B_1 \times \dotsb \times B_q \subset B_{|\rho}^q$.
By regularity, we thus have
\begin{align*}
		T_{\bB}(\alpha 1_B,\dots,\alpha 1_B)
		&= \alpha^t \big( \E_{x_0 \in B} 1_{B_{| 1-\rho}}(x_0) + O( \rho d ) \big) \\
		&= ( 1 + O(\rho d) ) \alpha^t \\
		&\geq \alpha^t /2.
\end{align*}
By~\eqref{eq:intg:multilinexp}, if we are not in the first case
of the proposition, then by the pigeonhole principle there must
exist a large average
\begin{align*}
	\alpha^t \ll |T_{\bB}(f_1,f_2,\dots,f_t)|
\end{align*}
where one of the functions $f_i : \Z_M \rightarrow [-1,1]$ is equal to $f_A$.
\end{proof}

The next step is to use the fact that
(twisted) local Gowers norms control the count 
of $\varphi$-configurations,
up to a small error.
This is the analog for general systems 
of complexity $1$ of Shao's~\cite[Proposition~4.1]{Shao:configs};
it is also very similar to
Green and Tao's generalized Von Neumann theorem
for bounded functions~\cite[Theorem~2.3]{GW:complexity}.

\begin{proposition}[Large average implies large Gowers norm]
\label{thm:intg:largeTtolargeU2}
	Let $\eta \in (0,1]$ be a parameter, and suppose that 
	$B_0,\dots,B_q$ are regular $d$-dimensional Bohr sets
	such that $B_i \leq_\rho B_{i-1}$ for all $i \in [q]$,
	where $\rho \leq c\eta^4 / d$.
	Suppose that $f_1,\dots,f_t : \Z_M \rightarrow [-1,1]$
	are such that
	\begin{align*}
		|T_\mathbf{B}(f_1,\dots,f_t)| \geq \eta.
	\end{align*}
	Then for every $i \in [t]$, there exist 
	$1 \leq k < \ell \leq q$ and 
	$a, b \in \Xi_i$ such that
	\begin{align*}
		\E_{u_0 \in B_0} \| f_i(u_0 + \cdot ) \|_{\boxtimes_{a,b}( B_k \times B_\ell )}^4 \geq \eta/2.
	\end{align*}
\end{proposition}

\begin{proof}
Let $i \in [t]$, and recall that $\psi$ is in exact $1$-normal form at $i$.
We may therefore find indices $1 \leq k < \ell \leq q$
and a partition $[t] \smallsetminus \{i\} = X_k \sqcup X_\ell$
into non-empty sets such that 
$\psi_i$ depends on the variables $x_k$ and $x_\ell$,
while for $j \in X_k$ (respectively $j \in X_\ell$),
$\psi_j$ depends at most on the variable $x_k$ (respectively $x_\ell$)
among those two variables.
We decompose vectors $x \in \Z^{q+1}$ accordingly 
as $x = (x_0,x_k,x_\ell,y)$ with $y \in \prod_{j \not\in \{0,k,l\}} B_j$,
and we may write $\psi_i(x_k,x_\ell,y) = a_k x_k + a_\ell x_\ell + \psi_i(0,0,y)$
with $a_k, a_\ell \in \Xi_i$.
Then\footnote{
We write $(B_j)_{j \in X}$ for $\prod_{j \in X} B_j$ in subscripts.}
\begin{multline*}
	\eta \leq 
	\big| \E_{ x_0 \in B_0, y \in (B_j)_{j \not\in \{0,k,\ell\} } }
	\E_{x_k \in B_k, x_\ell \in B_\ell} 
	f_i \big[ x_0 + \psi_i (x_k,x_\ell,y) \big]  \\
	\textstyle \times
	\prod_{j \in X_k} f_j \big[ x_0 + \psi_j (x_k,y) \big]
	\prod_{j \in X_\ell} f_j \big[ x_0 + \psi_j (x_\ell,y) \big] \big|.
\end{multline*}
We may rewrite the averaged function as $h(x_k,x_\ell) b_k (x_k) b_\ell (x_\ell)$,
where $h,b_k,b_\ell$ are functions depending on $x_0,y$ and $b_k,b_\ell$
are bounded by $1$.
By Hölder's inequality, followed by the 
box Van der Corput inequality~\eqref{eq:intg:boxVDC}, 
we thus have
\begin{align*}
	\eta^4 
	&\leq
	\big( \E_{ x_0 \in B_0, y \in (B_j)_{j \not\in \{0,k,\ell\} } } 
	\big| \E_{x_k \in B_k, x_\ell \in B_\ell } 
	h (x_k,x_\ell) b_{k}(x_k) b_{\ell}(x_\ell) \big|  \big)^4
	\\
	&\leq
	\E_{ x_0 \in B^{(0)}, y \in (B_j)_{j \not\in \{0,k,\ell\} } } 
	\big| \E_{x_k \in B_k, x_\ell \in B_\ell } 
	h (x_k,x_\ell)  b_{k}(x_k) b_{\ell}(x_\ell) \big|^4 
	\\
	&\leq \E_{ x_0 \in B_0, y \in (B_j)_{j \not\in \{0,k,\ell\} } }
	\| h \|_{\Box(B_k \times B_\ell) }^4 .
\end{align*}
Unfolding the definition of the box norm, and by regularity on the variable $x_0$, we have
\begin{align*}
	\eta^4 
	&\leq
	\E_{ x_0 \in B_0, y \in (B_j)_{j \not\in \{0,k,\ell\} } } 
	\E_{ x^{(0)},x^{(1)} \in B_k \times B_\ell }  
	\\
	&\phantom{ \E_{ x_0 \in B_0, y \in (B_j) } }
	\textstyle \prod_{ \omega\in \{0,1\}^2 }
	f_i ( x_0 + a_k x_k^{(\omega_k)} + a_\ell x_\ell^{(\omega_\ell)} + \psi_i (0,0,y) ) 
	\\
	&= \E_{ x_0 \in B_0 }
	\E_{ x^{(0)},x^{(1)} \in B_k \times B_\ell } 
	\textstyle \prod_{ \omega\in \{0,1\}^2 }
	f_i ( x_0 + a_k x_k^{(\omega_k)} + a_\ell x_\ell^{(\omega_\ell)} ) + O(\rho d). 	
\end{align*}
Refolding the definition of the $(a_k,a_\ell)$-twisted $U^2$ norm, this
concludes the proof, provided that $\rho \leq c\eta^4 / d$.
\end{proof}

We now wish to reduce the conclusion of the
previous proposition to the situation where $a=b=1$, that is,
when $f_A$ has a large (regular) local Gowers norm.
It turns out that such a reduction is 
always possible by a simple averaging argument,
together with an application of 
the Gowers-Cauchy-Schwarz inequality
to separate the translated functions 
arising from such a process.

\begin{proposition}
\label{thm:intg:untwistingU2}
Let $\eta \in (0,1]$ be a parameter.
Suppose that $B_0,B_1,B_2$
are regular $d$-dimensional Bohr sets such that $B_1,B_2 \leq_\rho B_0$,
and consider two other Bohr sets $\wt{B}_1 \leq_{\wt{\rho}} B_1$
and $\wt{B}_2 \leq_{\wt{\rho}} B_2$,
where $\rho,\wt{\rho} \leq c\eta^4 / d$.
Then for $f : \Z_M \rightarrow [-1,1]$ and $a,b \in \Xi$,
\begin{align*}
	\E_{u_0 \in B_0} \| f( u_0 + \cdot ) \|_{\boxtimes_{a,b}( B_1 \times B_2 )}^4 \geq \eta^4
	\Rightarrow
	\E_{u_0 \in B_0} \| f( u_0 + ab\, \cdot ) \|_{\boxtimes( \wt{B}_1 \times \wt{B}_2 ) }^4 \geq \eta^4 / 2
\end{align*}
\end{proposition}

\begin{proof}
Unfolding the definition of the twisted $U^2$ norm, we have
\begin{align*}
	\eta^4 \leq
	\E_{u_0 \in B_0} \E_{x^{(0)},x^{(1)} \in B_1 \times B_2}
	\textstyle\prod\limits_{\omega \in \{0,1\}^2 } f( u_0 + a x_1^{(\omega_1)} + b x_2^{(\omega_2)} ).
\end{align*}
By regularity, we now duplicate the variables 
$x_1^{(\eps)}$ into $x_1^{(\eps)} + b y_1^{(\eps)}$
with $y_1^{(\eps)} \in \wt{B}_1$, and
the variables $x_2^{(\eps)}$ into $x_2^{(\eps)} + a y_2^{(\eps)}$
with $y_2^{(\eps)} \in \wt{B}_2$, so that
\begin{align*}
	\eta^4 - O(\wt{\rho} d) &\leq
	\E_{u_0 \in B_0}
	\E_{x^{(0)},x^{(1)} \in B_1 \times B_2 }
	\E_{y^{(0)},y^{(1)} \in \wt{B}_1 \times \wt{B}_2 } \\
	&\phantom{\quad\quad}
	\textstyle\prod\limits_{ \omega \in \{0,1\}^2 }
	f\big( u_0 + a x_1^{(\omega_1)} + b x_2^{(\omega_2)} + ab( y_1^{(\omega_1)} + y_2^{(\omega_2)} ) \big)
	\\
	&=
	\E_{u_0 \in B_0}
	\E_{x^{(0)},x^{(1)} \in B_1 \times B_2 }
	\langle ( f( u_0 + a x_1^{(\omega_1)} + b x_2^{(\omega_2)} + ab S ) )_\omega 
	\rangle_{ \Box( \wt{B}_1 \times \wt{B}_2 ) },
\end{align*}
where $S : \wt{B}_1 \times \wt{B}_2 \rightarrow \Z_M$
is defined by $S (u_1,u_2) = u_1 + u_2$.
Applying successively the Gowers-Cauchy-Schwarz
inequality~\eqref{eq:intg:GCS} 
and Hölder's inequality, we obtain
\begin{align*}
	c\eta^{16}
	&\leq 
	\big( \E_{u_0 \in B_0}
	\E_{ x^{(0)} , x^{(1)} \in B_1 \times B_2 }
	\textstyle\prod\limits_{\omega \in \{0,1\}^2 }
	\| f( u_0 + a x_1^{(\omega_1)} + b x_2^{(\omega_2)} + ab S ) 
	\|_{ \Box(\wt{B}_1 \times \wt{B}_2) } \Big)^4
	\\
	&\leq
	\textstyle\prod\limits_{\omega \in \{0,1\}^2 }
	\E_{u_0 \in B_0}
	\E_{ x^{(0)} , x^{(1)} \in B_1 \times B_2 }
	\, \| f( u_0 + a x_1^{(\omega_1)} + b x_2^{(\omega_2)} + ab S ) 
	\|_{ \Box(\wt{B}_1 \times \wt{B}_2) }^4.
\end{align*}
By the pigeonhole principle, we may therefore find $\omega\in\{0,1\}^2$ such that
\begin{align*}
	c\eta^4 
	&\leq
	\E_{u_0 \in B_0} 
	\E_{ x^{(0)} , x^{(1)} \in B_1 \times B_2 }
	\| f( u_0 + ax_1^{(\omega_1)} + bx_2^{(\omega_2)} + ab S ) 
	\|_{\Box(\wt{B}_1 \times \wt{B}_2)}^4
	\\
	&=
	\E_{u_0 \in B_0} 
	\| f( u_0 + abS ) \|_{\Box(\wt{B}_1 \times \wt{B}_2)}^4 + O(\rho d),
\end{align*}
where we have used regularity in the variable $u_0$ in the last step.
The proposition follows from recalling Definition~\ref{thm:intg:twistedU2def}.
\end{proof}

At this point, we have reduced to a situation where
we may apply Shao's local inverse $U^2$ 
theorem~\cite[Theorem~3.2 and Lemma~5.1]{Shao:configs},
quoted below, to obtain a density increment.
The presence of a coefficient $m=ab$ calls for a minor 
variant\footnote{
Note also that Bohr sets on $\Z$ are used in that reference,
however this is only a cosmetic difference.
We actually quote a slightly weaker, but simpler, one-case consequence
of Shao's result to fluidify our argument.
}
of that result, which can however be effortlessly extracted out
of Shao's argument: we omit the proof.
Note also that in the proposition below,
we consider Bohr sets of $\Z_M$ as sets of integers
via the pullback of $\pi : [-M/2,M/2]_\Z \isom \Z_M$.

\begin{proposition}[Local inverse $U^2$ theorem~\cite{Shao:configs}]
\label{thm:intg:U2inv}
	Let $\eta \in (0,\tfrac{1}{2}]$ and $m \in \Xi \cdot \Xi$ be parameters.
	Suppose that $B_0,B_1,B_2$ are regular $d$-dimensional Bohr sets
	such that $B_1 \leq_\rho B_0$ and
	$B_2 \leq_\rho B_1$, where $\rho \leq c\eta^{12} / d$.
	Suppose also that $f: \Z_M \rightarrow [-1,1]$ 
	is such that $\E_{B_0} f = 0$ and
	\begin{align*}
		\E_{u_0 \in B_0}
		\| f(u_0 + m \, \cdot \,) \|_{\boxtimes(B_1 \times B_2)}^4 \gg \eta^4.
	\end{align*}
	Then there exists $u \in \Z$ and a regular Bohr set $B_3$ such that
	$u + m B_3 \subset B_0$ in $\Z$, and
	\begin{align*}
		d_3 \leq d + 1,
		\quad\text\quad
		\delta_3 \geq (\eta/d)^{O(1)} \delta_1,
		\quad\text\quad
		\E_{u + m B_3} f \geq c\eta^{12} .
	\end{align*}
\end{proposition}

We are now ready to combine the previous propositions
into our main density-increment statement,
which we then iterate to obtain Proposition~\ref{thm:intg:intgcplx1}.

\begin{proposition}[Main iterative proposition]
\label{thm:intg:mainiteration}
Suppose that $A$ is a subset of density $\alpha \in (0,\tfrac{1}{2}]$ 
of a regular $d$-dimensional Bohr set $B$ contained in $[-N,N]$.
Then either
\begin{enumerate}
	\item 	(Many $\varphi$-configurations) we have
			\begin{align*}
				\#\{ x \in [-N,N]^{q+1} : \varphi(x) \in A^t \} \geq (\alpha\delta/d)^{O(d)} N^{q+1},
			\end{align*}
	\item 	(Density increment) or there exists $u \in \Z$, $m \in \N$ 
			and a regular Bohr set $B'$ such that
			$u + m B' \subset B$ in $\Z$ and, writing $\alpha' = |A \cap (u + m B')|/|B'|$,
			\begin{align*}
				\alpha' \geq (1 + c\alpha^{12t-1}) \alpha,
				\quad\quad
				d' \leq d + 1,
				\quad\quad
				\delta' \geq (\alpha/d)^{O(1)} \delta.
			\end{align*}
\end{enumerate}
\end{proposition}

\begin{proof}
Write $\eta = \alpha^t$ and choose $\rho = c\eta^{12}/d$.
Let $B_0 = B$, and choose regular Bohr sets $B_1,\dots,B_q$
with $B_i = B_{i-1 | \rho_i}$ and $\rho_i \in [\rho/2,\rho]$,
so as to apply Proposition~\ref{thm:intg:multilinexp}.
Since $B_i \subset [-N,N]$ and $M > 2\|\varphi\| N$, 
for any $x \in B_0 \times \dotsb \times B_q$,
$\varphi(x)$ belongs to $A^t$ modulo $M$ if and only if it does in $\Z$.
Therefore, if we are in the first case of 
Proposition~\ref{thm:intg:multilinexp}, we have
\begin{align}
	\label{eq:intg:patterncount}
	\#\{ x \in [-N,N]^{q+1} : \varphi(x) \in A^t \}
	\geq c\alpha^t |B_0| \dots |B_q|
	\geq (\alpha\delta/d)^{O(d)} M^{q+1}.
\end{align}
In the second case, we deduce, by Proposition~\ref{thm:intg:largeTtolargeU2},
that there exist $i \in [t]$, $1 \leq k < \ell \leq q$ 
and twists $a,b \in \Xi_i$ such that, for $f_A = 1_A - \alpha 1_{B_0}$,
\begin{align*}
	\E_{u_0 \in B_0} \| f_A (u_0 + \,\cdot\,) \|_{\boxtimes_{a,b}(B_k \times B_\ell)}^4 \gg \eta^4.
\end{align*}
Via Proposition~\ref{thm:intg:untwistingU2}, we may assume instead that
\begin{align*}
	\E_{u_0 \in B_0} \| f_A (u_0 + ab\,\cdot\,) \|_{\boxtimes(\wt{B}_k \times \wt{B}_\ell)}^4 \gg \eta^4
\end{align*}
for regular dilates $\wt{B}_k = B_{k | \rho_k}$ and $\wt{B}_\ell = B_{\ell | \rho_\ell}$
with $\rho_k,\rho_\ell \in [\rho/2,\rho]$;
note that we have $\wt{B}_k \leq_{2\rho} \wt{B}_\ell$. 
Finally, an application of Proposition~\ref{thm:intg:U2inv} 
to $f_A$ yields a density increment of the desired shape.
\end{proof}

\textit{Proof of Proposition~\ref{thm:intg:intgcplx1}}.
As stated at the beginning of this section,
we use a parametrization $\varphi : \Z^{q+1} \twoheadrightarrow \Z^t \cap \Ker_\Q(V)$,
so that $\rk(\varphi) = \dim(\Ker_\Q V) = t-r$.
We embed $[-N,N]$ in a regular Bohr set 
$B^{(0)} \coloneqq B( \{1\} , \tfrac{c}{D} )$ of $\Z_M$,
where $c \in [1,2]$ and $M = DN$.
The set $A^{(0)} \coloneqq A$ then has density $\gg \alpha$ in $B^{(0)}$.
We now construct iteratively a sequence of regular Bohr sets $B^{(i)}$ 
of dimension $d_i$ and radius $\delta_i$ contained in $[-N,N]$, 
and a sequence of subsets $A_i$ of $B^{(i)}$ of density $\alpha_i$;
we also view $A_i$ as subsets of $\Z$ via the pullback
of $\pi : [-M/2,M/2]_\Z \isom \Z_M$.
At each step we apply Proposition~\ref{thm:intg:mainiteration}
to the set $A_i$, and in the second case of that proposition 
we define $A_{i+1}$ in $\Z$ by
\begin{align*}
	A_i \cap (u_{i+1} + m_{i+1} B_{i+1}) = u_{i+1} + m_{i+1} A_{i+1}.
\end{align*}
Writing $S_{\varphi}(Y) = \#\{ x \in [-N,N]^{q+1} : \varphi(x) \in Y^t \}$
for a set of integers $Y$, it follows from the linearity and
the presence of a shift variable in $\varphi$ that 
$S_{\varphi}(A) \geq S_{\varphi}(A_i)$ for every $i$.

From $\alpha_{i+1} \geq (1 + c\alpha_i^{12t-1})\alpha_i$ 
and a familiar geometric series summation~\cite[Chapter 6]{AG:acbook},
we deduce that the algorithm runs for at most $O(\alpha^{-12t+1})$ steps.
Iterating the dimension and radius bounds,
we also deduce that $d_i \ll \alpha^{-12t+1}$
and $\delta_i \geq \exp[ -C \alpha^{-12t+1} \log \alpha^{-1}]$.
Bounding crudely $\alpha^2 \log \alpha^{-1} \ll 1$,
we have therefore, in the first case of Proposition~\ref{thm:intg:mainiteration},
\begin{align}
\label{eq:intg:finallowerbound}
	\#\{ x \in [-N,N]^{q+1} : \varphi(x) \in A^t \}
	\geq \exp\big[ \! - C \alpha^{-24t} \big] \cdot N^{q+1}.
\end{align}
Since $\varphi$ has rank $t-r$, for each $y \in [N]^t$,
we have the multiplicity bound
\begin{align*}
	\#\{ x \in [-N,N]^{q+1} : \varphi(x) = y \} \ll N^{(q+1) - (t-r)}.
\end{align*}
Summing over values $y=\varphi(x)$ 
in~\eqref{eq:intg:finallowerbound},
we have therefore
\begin{align*}
	\#\{ y \in A^t : Vy = 0 \}
	\geq \exp\big[ \! - C \alpha^{-24t} \big] \cdot N^{t-r}.
\end{align*}
\qed

\bibliographystyle{amsplain}
\bibliography{primescplx1}

\bigskip

\textsc{\footnotesize Département de mathématiques et de statistique,
Université de Montréal,
CP 6128 succ. Centre-Ville,
Montréal QC H3C 3J7, Canada
}

\textit{\small Email address: }\texttt{\small henriot@dms.umontreal.ca}

\end{document}